\documentclass[matann2]{svjour1}

\usepackage{amsmath}
\usepackage{amsfonts}
\usepackage{latexsym}
\usepackage{mathrsfs}
\usepackage{array}
\usepackage{amscd}
\usepackage{a4}
\usepackage[mathcal]{euscript}
\usepackage{amssymb}
\usepackage{epsf, epic, eepic} 
\usepackage[alphabetic]{amsrefs} %% alphabetic citations
\usepackage{theorem}
\usepackage{enumerate}
\usepackage{relsize}
\usepackage{times}      %% text in times
\usepackage{txfonts}    %% (Elsevier) 
\usepackage{tikz}
\usetikzlibrary{arrows.meta}
\usetikzlibrary{decorations.pathmorphing}
\usetikzlibrary{positioning}
\usepackage{graphicx}
\usepackage[utf8]{inputenc}
\usepackage{color,soul,soulutf8}
\spnewtheorem*{remark*}{Remark}{\it}{\rm}
\spnewtheorem*{remarks*}{Remarks}{\it}{\rm}
\spnewtheorem*{firstproof*}{First proof}{\it}{\rm}
\spnewtheorem*{secondproof*}{Second proof}{\it}{\rm}
\DeclareMathAlphabet{\ams}{U}{msb}{m}{n}
\DeclareMathAlphabet{\goth}{U}{euf}{m}{n}
\def\Z{\ams{Z}}\def\R{\ams{R}}

\def\F{\ams{F}}
\def\X{\goth{X}}
\def\symn{S_{\kern-1pt n}}
\def\sym5{S_{\kern-1pt 5}}
\def\ss{\sigma}
\def\tt{\tau}
\def\ra{\rightarrow}

\def\St{\text{St}}
\def\Lk{\text{Lk}}
\def\partition{\,\,\begin{tikzpicture}
\draw[line width=0.7pt](0,0)--(0,0.2);
\draw[line width=0.7pt](0,0.1)--(0.175,0.1);
\end{tikzpicture}\,}

\newcommand{\Sym}[1]{S_{\kern-1pt {#1}}}
\newcommand{\rmod}{\vrule width 0mm height 0 mm depth
  0mm_R\mathbf{Mod}}
\title{Clique complexes of groups and order sheaves}

\author{Michael Bate %\thanks{} 
$\cdot$
Brent Everitt %\thanks{} 
$\cdot$
Sam Ford \thanks{The authors owe a debt of gratitude to Steve
Donkin for suggesting to them the idea behind the clique complex
and order sheaves. The second author thanks Paul Turner for
helpful discussions.}
}

\institute{
Department of Mathematics, 
University of York, 
Heslington, 
York
YO10 5GY, United Kingdom. 
\email{michael.bate@york.ac.uk,
brent.everitt@york.ac.uk,
sam.ford7@gmx.co.uk}. 
}

\titlerunning{Clique complexes of groups and order sheaves}
\authorrunning{Michael Bate, Brent Everitt, Sam Ford}

\begin{document}

\maketitle

\begin{abstract}
  We study the clique complex $X_G$ of a finite group: a simplicial
  complex whose faces are in one-to-one correspondence with the
  elements of the group, and arise via a decomposition of group
  elements into $p$-elements for primes $p$. The $1$-skeleton
  of this complex then has the classical Gruenberg-Kegel
  graph $\Gamma_G$ as a quotient and a finite simple group
  is determined up to isomorphism by its clique complex.
  The clique complex can also be endowed with
  a combinatorial sheaf --- called the order sheaf ---
  and the resulting homology can be computed
  for an arbitrary simplicial
  complex in terms of the homology of the links of vertices.
  The homology of the clique complex with coefficients in this
  order sheaf is then computed for some special classes of groups,
  in particular those for which the centralizers of elements
  are nilpotent. 
\end{abstract} 

\maketitle

%%%%%%%%%%%%%%%%%%%%%%%%%%%%%%%%%%%%%%%%%%%%%%%%%%%%%
%%%%%%%%%%%%%%%%%%%%%%%%%%%%%%%%%%%%%%%%%%%%%%%%%%%%%

\section*{Introduction}
\label{introduction}

The augmentation ideal of a finite group $G$ is the kernel of 
the map $\Z G\ra \Z$ induced by $g\mapsto 1$, where $\Z G$ is the 
integral group ring. The question of whether this ideal decomposes
as a $\Z G$-module has been studied in finite group theory, 
group cohomology, representation theory and homotopy theory.
In the 1970's, Gruenberg and Kegel introduced a graph $\Gamma_G$
whose connectedness properties control when these decompositions
happen. This Gruenberg-Kegel graph has subsequently been studied
for its connectedness, structure and as an invariant of $G$.

In this paper we study a higher dimensional analogue to the
Gruenberg-Kegel graph that we call the clique complex
$X_G$. This complex was first introduced by Steve Donkin
in \cite{MR2271571}*{page 54}.
It
is a simplicial complex whose faces arise from the decompositions
of group elements (Lemma \ref{lemma1} below) into mutually commuting
$p$-elements, for primes $p$. The vertices of $X_G$ are thus the 
elements of prime power order; two such, $g_0$ and $g_1$,
span an edge $g_0g_1$ when $g_0$ is a $p$-element, $g_1$ is a $q$-element,
with $p\not= q$, and $g_0,g_1$ commute. Three such $g_0,g_1,g_2$
span a $2$-face analogously, and so on. The resulting complex has
faces in 1-1 correspondence with the non-trivial elements of $G$. 

The clique complex fits into a hierarchy of invariants for groups,
with $X_G$ determining $\Gamma_G$, but not the other way around. 
For example, it is known that apart from a finite set of exceptional
simple groups,
there are only finitely many groups having the same 
Gruenberg-Kegel graph as a given finite simple group. On the other hand,
a finite simple group is determined up to isomorphism by
its clique complex --- see Theorem \ref{theorem:simple:groups}.

As we mentioned above, the connectedness of $\Gamma_G$ has structural
implications for $G$. As the connectedness of a complex $X$
is measured by the simplicial homology $H_0(X)$ in degree $0$,
we are naturally led to study the homology of the clique complex
more generally. It turns out, for a variety of reasons, to be more
fruitful to investigate the \emph{combinatorial sheaf\/}
homology of $X_G$. Combinatorial sheaves on posets,
simplicial complexes, and more generally small categories,
have found application recently to a number of areas:
Khovanov homology \cite{MR1740682} --- see also \cite{MR3276847} ---
and invariants of hyperplane arrangements \cite{MR4401823} are two such.
Sheaf homology often teases out more structure than ordinary homology does.
In the case of clique complexes equipped with what we call order sheaves,
there is an added, slightly unexpected, incentive: we will see 
--- Theorem \ref{clique:complex:theoremeuyrh9erw} below ---
that the sheaf homology can be computed with only partial
knowledge of the ordinary homology. 

The paper is organised as follows. The clique complex is defined
as an abstract simplicial complex in Section \ref{section:clique:general}, 
preceded in Section \ref{section:clique:complexes}
by the basic notions on simplicial complexes that we will need. 
There are two kinds of face of particular 
significance --- the facets and the isolated vertices --- and
these are described in Section \ref{section:clique:topological} 
along with the 
observation that the groups with $0$-dimensional clique
complex are precisely the $EPPO$-groups. A picture of a clique
complex ``up to conjugacy" is described in 
Section \ref{section:clique:schematic}
and this can sometimes be used to demonstrate connectedness
properties of $X_G$. Section \ref{section:clique:invariants}
discusses the extent to which a 
finite group is determined by various invariants such as
the Grunberg-Kegel graph, the spectrum and the clique
complex $X_G$.

Sections \ref{section:clique:complex:symmetric} to 
\ref{section:clique:complex:twocomponents} describe the clique
complex in more detail for the symmetric groups, the
nilpotent groups and the simple groups. In all cases, 
the dimension of $X_G$ is unbounded over each of these
families. Some of these families --- the symmetric groups,
alternating groups, nilpotent groups --- display an interesting
form of connectedness that we call mono-connectedness; see
Section \ref{section:clique:topological} for the definition.
On the other hand, among the groups of Lie type, 
mono-connectedness seems to be
quite rare; among the sporadic groups, some are
mono-connected and some not. The complete picture here
is unknown.

Section \ref{sheaves} contains the basics on combinatorial
sheaves and the homology of a simplicial complex with coefficients
in a combinatorial sheaf. We introduce the order sheaf on
a general simplicial complex and show how the resulting sheaf homology
reduces to the ordinary homology of the links of vertices
with (constant) coefficients a certain finite cyclic group.
If the simplicial complex is the clique complex of a group,
then the sheaf homology is determined by the ordinary homology of
the clique complexes of the centralizers of the elements of
prime power order. 

Section \ref{section:examples} 
implements these results for certain groups. 
It turns out that the nilpotent groups lend themselves very nicely
to this. We avail ourselves of the theory of shellings, which allows
us to compute the ordinary homology for the nilpotent groups,
but also the sheaf homology, as shellability is a property
of complexes that is inherited by links. As the homology of
the clique complexes of centralizers plays an important role in these
calculations, it is natural to then consider those groups
for which the centralizers of elements are themselves
nilpotent; these are the so-called $CN$-groups.
The paper concludes with a quite specific 
calculation of the sheaf homology
of the simple $CN$-groups. 

%%%%%%%%%%%%%%%%%%%%%%%%%%%%%%%%%%%%%%%%%%%%%%%%%%%%%
%%%%%%%%%%%%%%%%%%%%%%%%%%%%%%%%%%%%%%%%%%%%%%%%%%%%%

\section{Clique complexes}
\label{section:clique}

Clique complexes are simplicial complexes, 
and we start in Section \ref{section:clique:complexes}
with the basic notions that we will need; see
\cite{MR1402473}*{Section 2.3} or 
\cite{MR0210112}*{Section 3.1} for more details.
Section \ref{section:clique:general} defines the complex for a
general finite group and 
Section \ref{section:clique:topological}
lists its basic properties. As the size of $G$ gets larger,
it becomes harder to draw the clique complex of $G$.
Section \ref{section:clique:schematic} describes
an approximate picture, called a schematic, that gives
a pictorial flavour of $X_G$, even for quite large groups.
(For example, Figure \ref{figure2} shows the schematic 
for the Mathieu group $M_{24}$, of order 
$2^{10}\cdot 3^3\cdot 5\cdot 7\cdot 11\cdot 23$.)
Finally, Section \ref{section:clique:invariants} shows how
the clique complex fits into a hierarchy of invariants,
and that a finite simple group is determined by its 
clique complex.

%Sections
%\ref{section:clique:complex:symmetric}-\ref{section:clique:complex:twocomponents}
%give more
%detailed structure of the clique complex for some classes of groups
%--- especially the simple groups --- with a particular focus on their dimension
%and connectedness. The section then ends with a conjecture.

%%%%%%%%%%%%%%%%%%%%%%%%%%%%%%%%%%%%%%%%%%%%%%%%%%%%%

\subsection{Simplicial complexes}
\label{section:clique:complexes}

An abstract simplicial complex $X$ --- with vertex 
set $X_0$ --- is 
a set of distinguished subsets of $X_0$, with the property 
that $\{x\}\in X$ for all $x\in X_0$, and
if $\ss\in X$ and $\tt\subseteq\ss$ then $\tt\in X$. 
If $\ss\in X$ with $|\ss|=i+1$, then 
$\ss$ is called an $i$-(dimensional) face of $X$.  
Write $X_i$ for the set of $i$-faces and $X^{(i)}$
for the \emph{$i$-skeleton\/}: the union of
the $\ell$-faces, for $0\leq \ell\leq i$.
We allow the empty complex, with empty vertex set. 

A face $\ss$ that is maximal under inclusion is called a 
\emph{facet\/}
and a complex where all the facets have the same dimension
is said to be \emph{pure\/}. 
The dimension of $X$ is the maximal $i$ for 
which
there exists an $i$-facet.
An $i$-simplex
$\overline{\ss}$ is the
simplicial subcomplex obtained by considering an $i$-face $\ss$ and
all of its subsets; in particular, if $\ss$ is a 
facet then $\overline{\ss}$ is
a maximal subsimplex.

A \emph{subcomplex\/} $Y\subseteq X$ has vertex set $Y_0\subseteq X_0$
and faces those $\ss\in X$ where $\ss\subseteq Y_0$. 
A simplicial map
$f:X\ra Y$ of simplicial complexes is a map $f:X_0\ra Y_0$ such that
$f(\ss)\in Y$ for all $\ss\in X$. 
A simplicial complex $X$ is \emph{connected\/} when the $1$-skeleton 
$X^{(1)}$ is a connected simplicial graph. The \emph{connected
components\/} are the maximal connected subcomplexes. We take
the empty complex to be connected. 

A finite $d$-dimensional simplicial complex $X$
(and all our complexes will be finite)
has a geometric realisation $|X|$ in $\R^{2d+1}$, consisting
of geometric simplices suitably glued together. This allows
us to attach topological adjectives to $X$: it is
contractible, deform retracts, etc., when this is
the case for $|X|$.
The faces of $X$ form a poset $P_X$ under inclusion,
called the face poset of $X$.

If $\ss$ is a face, then the star $\St_\ss$ is
the subcomplex $\St_\ss=\{\tau\in X:\tau\cup\ss\in X\}$
and the link is the subcomplex
$\Lk_\ss=\{\tau\in X:\tau\cup\ss\in X,\tau\cap\ss=\varnothing\}
=\{\tau\in\St_\ss:\tau\cap\ss=\varnothing\}$.
If $\ss=x$ is a vertex, then the star $\St_x$
deform retracts onto $x$; in particular, $\St_x$
is contractible. 

If $X$ and $Y$ are complexes, then their
join $X\ast Y$ is the complex with vertex set $X_0\cup Y_0$
and faces the $\ss\cup\tt$, where $\ss$ is a face of $X$ and $\tt$
is a face of $Y$. Both $X$ (the faces of the form $\ss\cup\varnothing$) and
$Y$ (the faces $\varnothing\cup\tt$) are subcomplexes of the join
and $\dim X*Y=\dim X+\dim Y+1$.

%%%%%%%%%%%%%%%%%%%%%%%%%%%%%%%%%%%%%%%%%%%%%%%%%%%%%

\subsection{The clique complex $X_G$ of a finite group $G$}
\label{section:clique:general}

The construction starts with the
following decomposition:

\begin{lemma}
\label{lemma1}
Let $G$ be a finite group. If $g\in G$ has order $p_0^{m_0}\ldots p_i^{m_i}$,
with the $p_j$ distinct primes, then $g$ has
a unique decomposition
$g=g_0g_1\ldots g_i$,
with $g_j$ having order 
a power of $p_j$ %(and so is a \emph{$p_i$-element})
and $g_jg_k=g_kg_j$ for all 
$j,k$. 
\end{lemma}

\begin{proof}
To see the decomposition, write the order of $g$
as $p_0^{m_0}m$, and let $a,b$ be integers with $ap_0^{m_0}+bm=1$.
Then:
$$
g=g^{bm}g^{ap_0^{m_0}}:=g_0g',
$$
with $g_0,g'$ commuting, 
$g_0$ a $p_0$-element and  
$g'$ of order dividing $m=p_1^{m_1}\ldots p_i^{m_i}$.
Now continue with $g'$ and the prime $p_1$.
To see the uniqueness, 
suppose that 
$$
g = h_0h_1\ldots h_i
$$
is another such decomposition.
Then let 
$$
h = h_0^{-1} g = h_1\ldots h_i = (h_0^{-1}g_0) g_1 \ldots g_i.
$$
Since $h_0$ commutes with $g$ by hypothesis, 
and since each $g_j$ is constructed as a power of $g$, 
we have that $h_0$ commutes with every $g_j$.
But this means that $h_0^{-1}g_0$ is a $p_0$-element, 
and thus $h$ has order divisible by $p_0$ unless $h_0^{-1}g_0 = 1$.
The element $h = h_1\ldots h_i$ has order prime to $p_0$,
so $h_0 = g_0$, and now an easy induction completes 
the uniqueness proof.
\qed
\end{proof}

The \emph{clique complex $X_G$\/} of $G$ is the simplicial complex
whose vertices are 
the $1\not=g\in G$ of prime power order and with
$\{g_0,g_1,\ldots, g_i\}$
an $i$-face if and only if $g=g_0g_1\ldots g_i$ is the 
decomposition of Lemma \ref{lemma1}
for $g$. We will just write $g_0g_1\ldots g_i$ for
a face from now on.

There is thus
a 1-1 correspondence between the non-empty faces of $X_G$ and the nontrivial
elements of $G$, and
we transfer simplicial terminology to $G$ and its elements. A group $G$
thus has a dimension, is connected $\ldots$; elements have dimension, 
are faces, facets $\ldots$; and so on. 
We emphasise that $X_G$ is not just an abstract simplicial complex,
but an simplicial complex whose vertices correspond
(and hence can be labelled by) the $p$-elements of $G$. 
We will write $X_G^\circ$ for the abstract simplicial complex
underlying $X_G$.

The word ``clique'' is chosen for the following reason:
if $\Gamma$ is a simplicial graph, then the clique complex
of $\Gamma$ has $1$-skeleton $\Gamma$ and the vertices
$\{x_0,x_1,\ldots,x_i\}$ span an $i$-face if and only if any two of them are
joined by an edge. Thus, the $i$-faces are spanned by the vertices
that form a complete subgraph --- or clique --- in $\Gamma$.
If $\Gamma$ is the $1$-skeleton of $X_G$
then $\{g_0,g_1,\ldots, g_i\}$
form a clique if and only if
the product $g=g_0g_1\ldots g_i$ is a decomposition
of the form given in Lemma \ref{lemma1}.
The clique complex of a group is thus the clique complex
of its $1$-skeleton.

An isomorphism $f:X_G\ra X_H$ of clique complexes is 
an isomorphism $X_G^\circ\ra X_H^\circ$ 
between the underlying simplicial complexes
with the property that if $g\in G$ is an element of order
$p^n$ then $f(g)\in H$ is an element of order $p^n$. It is
clear that an isomorphism $G\ra H$ of groups induces an
isomorphism $X_G\ra X_H$ between their clique complexes. 
The converse is not true --- 
see Section \ref{section:clique:invariants} for an example
of non-isomorphic groups $G,H$ having isomorphic clique complexes.
An isomorphism $f:X_G\ra X_H$ does however induce
a bijection between $G$ and $H$ with the property that if $g,h\in G$ 
lie in the same face of $X_G$ with $g\cap h=\varnothing$
(which happens exactly when $g,h$ have relatively prime orders
and commute), then $f(gh)=f(g)f(h)$.

\vspace{1em}

The clique complex is partially functorial: 
an \emph{injective\/} homomorphism
$\varphi:H\ra G$ induces a simplicial map
$f_\varphi:X_H\ra X_G$ given by
$f_\varphi:g_0g_1\ldots g_i\mapsto \varphi(g_0)\varphi(g_1)\ldots \varphi(g_i)$.
%(If $\varphi$ has non-trivial kernel then this map
%will send vertices of $X_H$ to non-vertices of $X_G$;
%other way around?)
In particular, if $H$ is a subgroup of $G$, then $X_H$ is a
subcomplex of $X_G$.

For any $h\in G$ and $g=g_0g_1\ldots g_i$, the conjugate
$g^h:=hgh^{-1}$
decomposes as
$g^h=g_0^hg_1^h\ldots g_i^h$.
The group $G$ thus acts by simplicial automorphisms/conjugation on the clique
complex $X_G$.
This means that simplicial properties of elements depend only on the conjugacy
class of the element: the dimension of elements is constant across
a conjugacy class; all elements of a conjugacy class are either
facets or non-facets; and so on.

There are certain subcomplexes of the clique complex that will
be needed in some of our later calculations. If $p$ is a prime 
dividing the order of $G$ then the \emph{$p$-restricted 
clique complex} $X_G^p$ consists of the faces $g_0g_1\ldots g_i$
of $X_G$ where none of the $g_i$ are $p$-elements. It is easy
to see that this is a subcomplex of $X_G$.

%%%%%%%%%%%%%%%%%%%%%%%%%%%%%%%%%%%%%%%%%%%%%%%%%%%%%

\subsection{Topological properties of $X_G$}
\label{section:clique:topological}

A vertex of a simplicial complex is \emph{isolated\/} when it 
is not properly contained in any face; equivalently, an isolated
vertex is a $0$-dimensional facet. 

\begin{lemma}
  \label{lemma2}
  Let $X_G$ be the clique complex of $G$ and $g\in G$ a face.
  \begin{enumerate}[(i)]
  \item The dimension of $g$ is $\omega(g)-1$, where $\omega(g)$,
    the prime omega function, is the number of distinct prime
    divisors of the order of $g$.
  The dimension of $X_G$ is $\max_{g\in G}\{\omega(g)\}-1$. 
  \item The element $g$ is a facet if and only if for any prime $p$
    that divides the order of the centralizer of $g$, the prime
    $p$ also divides the order of $g$.
  \item The element $g$ is an isolated vertex if and only if
    $g$ is a $p$-element for some prime $p$ and the centralizer $C_g$
    of $g$ is a $p$-group.
  \end{enumerate}
\end{lemma}

\begin{proof}
The dimensions of $g$ and $G$ follow immediately
from the decomposition in
Lemma \ref{lemma1}. Given such a decomposition
$g=g_0g_1\ldots g_i$, the $g_j$ are powers of $g$, so that an $h\in G$
centralizes $g$ exactly when $h$ centralizes $g_j$ for all
$1\leq j\leq i$. Thus, $g$ is \emph{not\/} a facet
iff $hg_0g_1\ldots g_i$ is a face for some vertex $h$
iff $h$ centralizes $g$ and is a $p$-element with $p\not=p_j$
for any $j$
iff $p\not=p_j$ divides the order of the centralizer of $g$
iff $p$ divides the order of the centralizer of $g$ but
does not divide the order of $g$. Part (ii) follows.
Part (iii) follows from part (ii) because an isolated vertex
is precisely a $0$-dimensional facet. 
\qed
\end{proof}

\begin{lemma}
  \label{lemma4}
\begin{enumerate}[(i)]
\item $X_G$ is $0$-dimensional if and only if all 
non-trivial elements 
of $G$ have prime power order.
\item If $G,H$ have relatively prime orders then 
$X_{G\times H}$ is isomorphic to the join $X_G\ast X_H$.
\end{enumerate}
\end{lemma}

\begin{proof}
For part (i), the elements of $G$ all have prime
power order exactly when every element is a vertex
and this in turn happens when $X_G$ is $0$-dimensional.
For part (ii),
if $g=g_0g_1\ldots g_i$ and
$h=h_0h_1\ldots h_j$ are the decompositions
of $g\in G$ and $h\in H$, then 
$g_0g_1\ldots g_ih_0h_1\ldots h_j$ is the decomposition
of $(g,h)\in G\times H$.
\qed
\end{proof}

Groups in which every non-trivial element has prime power 
order are called 
\emph{EPPO-groups\/} or \emph{CP-groups}. Their study was 
started by Higman in \cite{MR0089205}. In 
\cite{MR0136646}*{Theorem 16}, Sukuki shows that the
non-Abelian simple EPPO-groups are
the $\mathrm{PSL}_2(q)$ for $q=5,7,8,9,17$, as well as
$\mathrm{PSL}_3(4)$, and the Suzuki groups 
$\mathrm{Suz}(8)$ and $\mathrm{Suz}(32)$.

Clique complexes are often disconnected, but 
for some families of groups this
disconnectedness happens in quite a
particular way: the non-isolated
vertices of $X_G$ form a single connected component.
Put loosely, $X_G$ 
consists of a connected component that is surrounded by a 
cloud of isolated vertices. We will call such groups
\emph{mono-connected\/}, and 
the connected component of the non-isolated vertices
the \emph{core\/} $X_\circ$ of $X=X_G$.
We will see below that $X_G$ is mono-connected
for all but finitely many $\Sym{n},A_n$, 
as well as the nilpotent groups and some of the sporadic
simple groups. 

When computing sheaf homology later in the paper, the 
results will be driven by the ordinary homology of the 
links of vertices. It turns out that these links are
\emph{almost\/} clique complexes too:

\begin{lemma}
\label{lemma:links}
    Let $g\in G$ be a $p$-element with centralizer $C_g$
    and let $\Lk_g$ be the link of the vertex $g$ in $X_G$.
    Then $\Lk_g=X_{C_g}^p$, the $p$-restricted clique complex
    of the centralizer.
\end{lemma}

\begin{proof}
The link $\Lk_g$ consists of the faces $g_0g_1\ldots g_i\in X_G$
such that $gg_0g_1\ldots g_i$ is also a face of $X_G$. But this
is precisely the subcomplex $X_{C_g}^p$.
\qed
\end{proof}

%%%%%%%%%%%%%%%%%%%%%%%%%%%%%%%%%%%%%%%%%%%%%%%%%%%%%

\subsection{The schematic $\X_G$ of $X_G$}
\label{section:clique:schematic}

One can draw pictures of clique complexes,
with the element $g=g_0g_1\ldots g_i$ appearing as
an $i$-dimensional geometric simplex whose faces
correspond to the elements $g_{i_0}g_{i_1}\ldots g_{i_\ell}$,
for $0\leq i_0<i_1<\cdots<i_\ell\leq i$.
The resulting picture of $X_G$ carries a great deal of redundancy
however. Instead, we will draw a ``picture up to conjugacy'',
or schematic.

The schematic $\X_G$ has faces in 1-1 correspondence with the 
conjugacy classes of non-identity elements of $G$. The vertex set
is the set of conjugacy classes of elements of prime power order. 
If $K$
is a conjugacy class and $g\in K$ a representative, then let
$g=g_0g_1\ldots g_i$ be the decomposition of Lemma \ref{lemma1}.
Then $\X_G$ has an $i$-dimensional geometric simplex
whose $\ell$-faces correspond to the conjugacy classes
of the elements $g_{i_0}g_{i_1}\ldots g_{i_\ell}$,
for $0\leq i_0<i_1<\cdots<i_\ell\leq i$.
In particular, the face corresponding to $K$ is
spanned by the vertices $K_0,K_1,\ldots,K_i$, where
$K_j$ is the conjugacy class of $g_j$ for $0\leq j\leq i$.

We are using the same terminology for a schematic as we would use
for a simplicial complex, but 
the schematic is not in general a simplicial complex:
an $i$-face has $i+1$ distinct vertices, but
the same set of $i+1$ vertices may span more than one
$i$-face. For example, the Mathieu group $M_{24}$ has a single conjugacy
class of elements of order $3$ and a single class of elements of order $5$,
but two conjugacy classes of elements of order $15$. The schematic 
--- see Figure \ref{figure2} --- has two edges, $15A$ and $15B$, joining
the vertices $3A$ and $5A$.

\begin{figure}[t]
\begin{tikzpicture}
\draw [white] (0,0)--(\textwidth,0);
\node at (3,4.5){\includegraphics[scale=0.35]{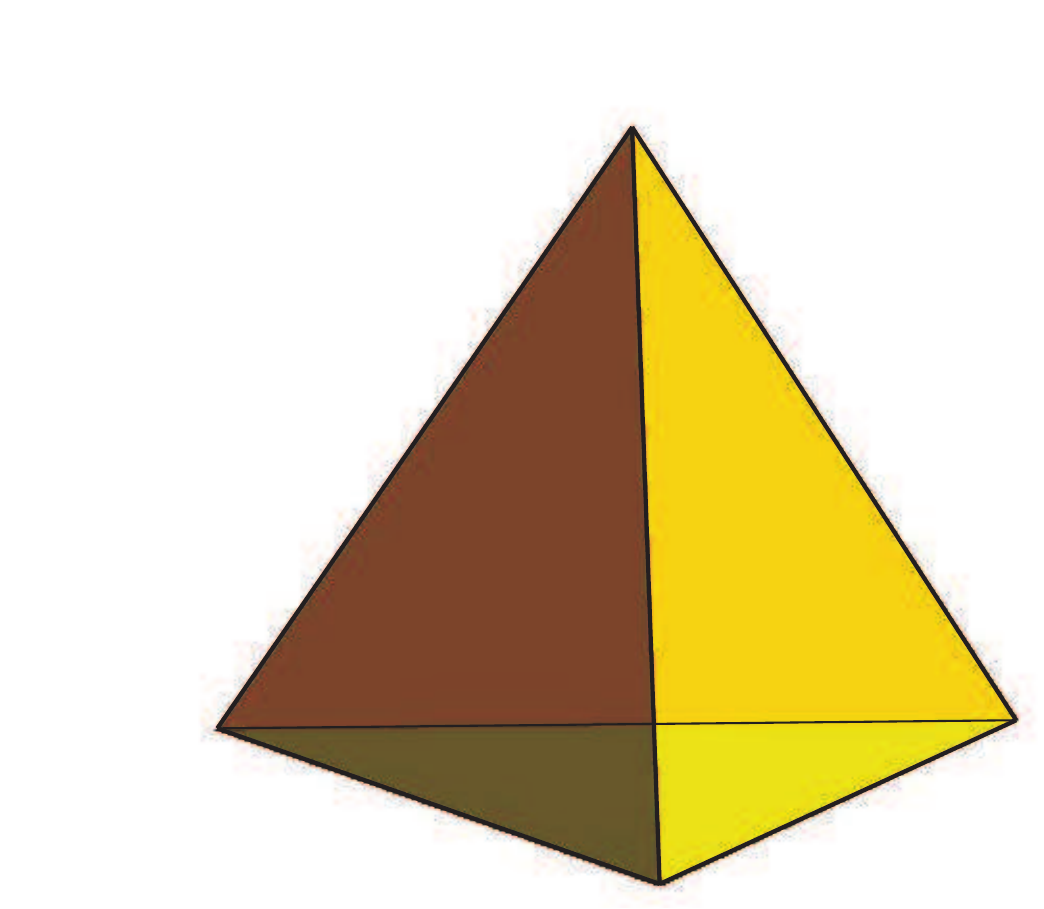}};
\node[color=red] at (3.65,6.65){$7$};
\node[color=red] at (3.8,1.75){$3^2$};
\node[color=red] at (1,2.9){$2^5$};
\node[color=red] at (6.15,2.9){$5^2$};
\node[color=blue] at (2.25,2.25){$6\cdot 2^2$};
\node[color=blue] at (5.15,2.25){$5^2\cdot 3^2$};
\node[color=blue] at (1.95,4.5){$7\cdot 2^5$};
\node[color=blue] at (5.35,4.5){$7\cdot 5^2$};
\node[color=blue] at (4.5,3.1){$5^2\cdot 2^5$};
\node[color=blue] at (4.15,4.2){$7\cdot 3^2$};
\node[color=yellow] at (2.75,3.75){$7\cdot 6\cdot 2^2$};
\node at (10,4.3){\includegraphics[scale=1]{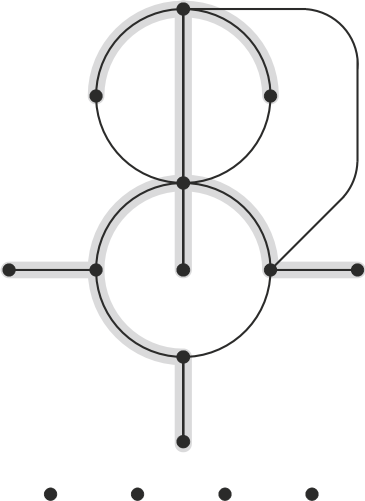}};
\node[color=red] at (12.2,0.45){$9$};
\node[color=red] at (10.7,0.45){$8$};
\node[color=red] at (7.8,0.45){$4^2$};
\node[color=red] at (9.3,0.45){$4\cdot 2^2$};
\node[color=red] at (13.2,4){$7$};
\node[color=red] at (6.75,4){$3^3$};
\node[color=red] at (10,8.7){$5$};
\node[color=red] at (8.25,6.9){$4$};
\node[color=red] at (11.8,6.9){$2^2$};
\node[color=red] at (10.2,5.7){$3$};
\node[color=red] at (11.1,4){$2$};
\node[color=red] at (8.9,4){$2^3$};
\node[color=red] at (10,0.75){$2^4$};
\node[color=red] at (10,2.8){$3^2$};
\node[color=red] at (10,3.7){$4\cdot 2$};
\node[color=blue] at (8.25,4.85){$3\cdot 2^3$};
\node[color=blue] at (11.75,4.85){$3\cdot 2$};
\node[color=blue] at (8.45,3){$6$};
\node[color=blue] at (11.65,3){$3^2\cdot 2$};
\node[color=blue] at (7.75,4.2){$6\cdot 3$};
\node[color=blue] at (12.25,4.2){$7\cdot 2$};
\node[color=blue] at (8.3,7.85){$5\cdot 4$};
\node[color=blue] at (11.75,7.85){$5\cdot 2^2$};
\node[color=blue] at (8.4,6){$4\cdot 3$};
\node[color=blue] at (11.65,6){$3\cdot 2^2$};
\node[color=blue] at (10.35,6.9){$5\cdot 3$};
\node[color=blue] at (10.5,4.6){$4\cdot 3\cdot 2$};
\node[color=blue] at (10.35,1.75){$6\cdot 2$};
\node[color=blue] at (13.35,7){$5\cdot 2$};
%
%\draw[help lines,blue,line width=.6pt,step=1] (0,0) grid
%(\textwidth,9); % blue grid boxes; comment out when done
%
\end{tikzpicture}
\caption{The $3$-face in the schematic $\X_{\symn}$ corresponding to
  the elements with cycle structure $(7,5^2,6,2^2)$ \emph{(left)}
  and the schematic $\X_{\Sym{9}}$ \emph{(right)}.}
\label{figure1}
\end{figure}
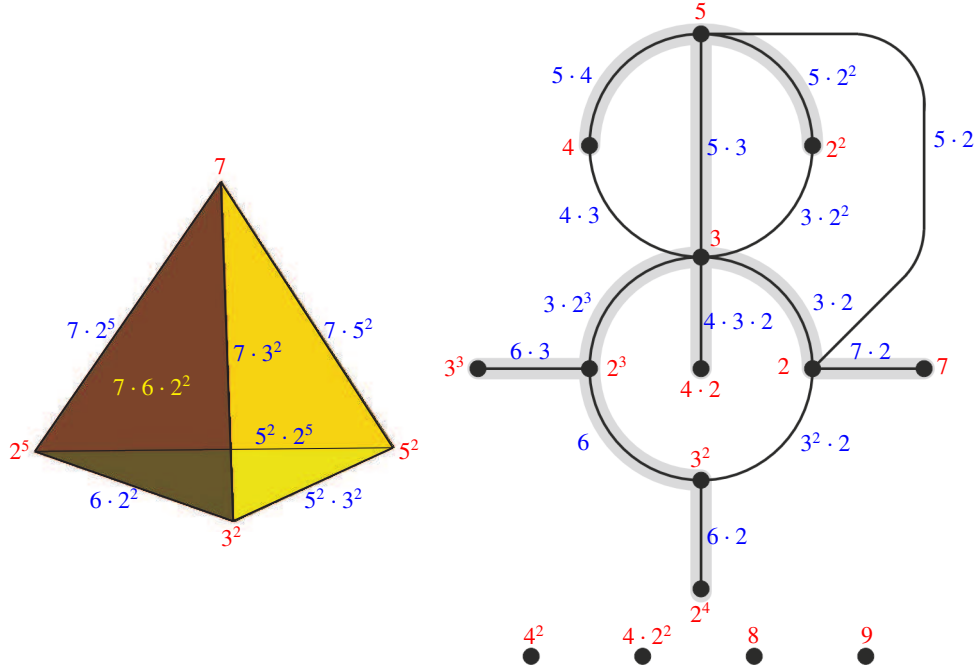

An $i$-face $K$ in the schematic $\X_G$ gives an $i$-face in the 
clique complex $X_G$ in the following way: 
write $K=K_{01\ldots i}$, so that for each subset
$\mu\subseteq\{0,1,\ldots,i\}$, there is a 
$(|\mu|-1)$-face of $K$. Write $K_\mu$ for the
corresponding conjugacy class.
Fix a face $\mu_0$ and choose
an element $g_{\mu_0}\in K_{\mu_0}$. Then for all $\mu$ there
is a $g_\mu\in K_\mu$ (with $g_\mu=g_{\mu_0}$ when $\mu=\mu_0$)
such that the $g_\mu$ are faces of an $i$-face of $X_G$.

In particular, if $K_{01}$ is an edge of $\X_G$, then for
any $g_0\in K_0$ there is a $g_1\in K_1$, with an edge
$g_0g_1\in X_G$ connecting them. This process can be iterated.
If $\goth{T}$ is a subtree of $\X_G$ with vertices the
conjugacy classes $K_1,K_2,\ldots,K_m$, then there are
elements $g_i\in K_i\,(1\leq i\leq m)$ such that $g_1,g_2,\ldots,g_m$
form the vertices of a subtree $T$ of $X_G$ that is isomorphic
to $\goth{T}$.

\vspace{1em}

Figure \ref{figure1} gives two examples in the symmetric group $\symn$,
where the conjugacy classes are parametrized by integer
partitions of $n$. 
For the first, let $g\in\symn\,(n\geq 27)$
be an element with cycle structure
$(7,5^2,6,2^2)$, omitting cycles of length $1$. Then by
Lemma \ref{lemma:symmetric:decomposition}, $g$
has a decomposition $g=g_0g_1g_2g_3$ with $g_0$ a $7$-cycle,
$g_1$ two disjoint $5$-cycles, $g_2$ two disjoint $3$-cycles
and $g_3$ five disjoint $2$-cycles (the $6$-cycle
$(a,b,c,d,e,f)$ decomposes as a product of
$(a,e,c)(d,b,f)$ and $(a,d)(b,e)(c,f)$).
The schematic for $X_{\symn}$
contains the $3$-simplex shown on the left of Figure
\ref{figure1}, where the vertices are labelled by the conjugacy classes
of elements with cycle structures
$(7),(5^2),(3^2)$ and $(2^5)$; the edges by
$(6,2^2),(5^2,3^2),(5^2,2^5),(7,2^5),\ldots\text{etc}\ldots$ 
We have abbreviated
$(5^2,3^2)$ to $5^2\cdot 3^2$, and so on. Only one of the $2$-faces
is labelled; the $2$-face on the right corresponds to the
class $7\cdot 5^2\cdot 3^2$, the back face to $7\cdot 5^2\cdot 3^2$
and the bottom to $5^2\cdot 6\cdot 2^2$.

The right of Figure \ref{figure1} shows the schematic for $\Sym{9}$,
which is the largest $1$-dimensional symmetric group
--- $\Sym{10}$ contains the $2$-face $(1,2)(3,4,5)(6,7,8,9,10)$.
The schematic is drawn using Lemmas
\ref{lemma:symmetric:decomposition}-\ref{lemma:symmetric:facets} of
Section \ref{section:clique:complex:symmetric} below.
There are four conjugacy classes of isolated vertices:
those with cycle structures $(9),(8),(4^2)$ and $(4,2^2)$.

\vspace{1em}

The connected components of the schematic can sometimes
correspond to connected components of the clique complex:

\begin{lemma}
  \label{lemma5}
  Let $\goth{C}$ be a connected component of $\X_G$ 
  with vertices $K_1,K_2,\ldots,K_j$ and let 
  $C$ be the subcomplex of $X_G$ spanned
  by the vertices in $\bigcup K_i$.
  Suppose that
  there are group elements $g_1\in K_1,g_2\in K_2,\ldots,g_j\in K_j$ such that
  $g_1,g_2,\ldots,g_k$ lie in the same connected component of $C$
  and the centralizers $C_{g_1},C_{g_2}\ldots,C_{g_j}$ generate $G$.
  Then $C$ is connected.
\end{lemma}

\begin{figure}[t]
\begin{tikzpicture}
\draw [white] (0,0)--(\textwidth,0); 
\begin{scope}[xshift=0mm,yshift=-3mm]
\node at (7,4.5){\includegraphics[scale=1]{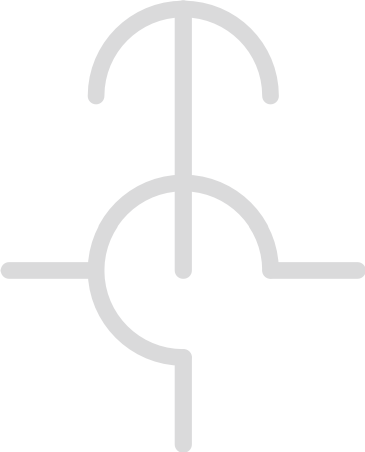}};
\node %[fill=white,inner sep=4pt] 
      at (7,8.15){$\scriptstyle{(1,2,3,4,5)}$};
\node %[fill=white,inner sep=4pt] 
      at (5.5,6.65){$\scriptstyle{(6,7,8,9)}$};
\node %[fill=white,inner sep=4pt] 
      at (8.5,6.65){$\scriptstyle{(6,7)(8,9)}$};
\node %[fill=white,inner sep=4pt] 
      at (7,5.2){$\scriptstyle{(7,8,9)}$};
\node %[fill=white,inner sep=4pt] 
      at (5.5,3.75){$\scriptstyle{(1,2)(3,4)(5,6)}$};
\node %[fill=white,inner sep=4pt] 
      at (3.4,3.75){$\scriptstyle{(1,5,3)(2,6,4)(7,8,9)}$};
\node %[fill=white,inner sep=4pt] 
      at (7,3.75){$\scriptstyle{(1,2,3,4)(5,6)}$};
\node %[fill=white,inner sep=4pt] 
      at (8.5,3.75){$\scriptstyle{(1,2)}$};
\node %[fill=white,inner sep=4pt] 
      at (10.1,3.75){$\scriptstyle{(3,4,5,6,7,8,9)}$};
\node %[fill=white,inner sep=4pt] 
      at (7,2.3){$\scriptstyle{(1,5,3)(2,6,4)}$};
\node %[fill=white,inner sep=4pt] 
      at (7,0.8){$\scriptstyle{(1,4)(2,5)(3,6)(7,8)}$};
\end{scope}
%
%\draw[help lines,blue,line width=.6pt,step=1] (0,0) grid
%(\textwidth,9); % blue grid boxes; comment out when done
%
\end{tikzpicture}
\caption{A subtree $T$ of $X_{\Sym{9}}$ given by the spanning
tree $\goth{T}$ of $\X_{\Sym{9}}$ in Figure \ref{figure1}.}
\label{figure1a}
\end{figure}

Connected components of the schematic do not give connected components
of the clique complex without the extra conditions given in the lemma:
for example, if $\goth{C}$ is an isolated vertex of $\X_G$ then 
the subcomplex $C$ of the lemma is a conjugacy class of isolated 
vertices in $X_G$.

\begin{proof}
Let $C^\circ$ be the connected component of $C$ that contains 
$g_1,g_2,\ldots,g_k$. In the action of $G$ on $X_G$ by conjugation,
the centralizer $C_{g_i}$ fixes $g_i$ for each $i$, 
hence each $C_{g_i}$ sends $C^\circ$ to itself. By the generating property,
$G$ then sends $C^\circ$ to itself. If $g\in \bigcup K_i$ then
there is an element of $G$ sending one of the $g_i$ to $g$,
and so $g\in C^\circ$ and hence $C^\circ=C$.
\qed
\end{proof}

We can use Lemma \ref{lemma5} to show that $\Sym{9}$ is mono-connected.
Let $\goth{C}$ be the single large connected component 
of $\X_{\Sym{9}}$ and let $\goth{T}$ be the 
shaded spanning tree for this component shown in Figure \ref{figure1}.
Then there are elements of $\Sym{10}$ that form the
vertices of a subtree $T$ of the clique complex $X_{\Sym{9}}$
that is isomorphic to $\goth{T}$ --- see Figure \ref{figure1a}.
Let these elements be the $g_i$ of Lemma \ref{lemma5}. Clearly
$(1,2),(2,3),\ldots,(5,6)$ lie in the centralizer of 
$(7,8,9)$, while $(6,7),(7,8),(8,9)$ are in the centralizer
of $(1,2,3,4,5)$. Thus, these two centralizers generate
$\Sym{9}$ and the subcomplex $C$ of Lemma \ref{lemma5} is
connected. The conclusion is that $\Sym{9}$ is mono-connected
with the elements of cycle structures $4\cdot 2^2,4^2,8$ and $9$
being isolated vertices, and all the other elements lying in
a single connected component. We show below, in a different way,
that $\Sym{n}\,(n\not=5,6,7)$ is 
mono-connected. 

%%%%%%%%%%%%%%%%%%%%%%%%%%%%%%%%%%%%%%%%%%%%%%%%%%%%%

\subsection{A hierarchy of invariants}
\label{section:clique:invariants}

The clique complex lies at the apex of a pyramid of invariants
for finite groups. Invariants further up the pyramid determine
those lower down, and the higher up one goes, the more
discerning the invariants become. By invariant, we mean an object
$\Omega_G$ --- a set or a graph or a complex --- attached to $G$,
and such that a group isomorphism $G\cong H$ induces an
isomorphism $\Omega_G\cong\Omega_H$. If $\Omega_G$ is a set of 
integers, then an isomorphism $f:\Omega_G\ra\Omega_H$ is a bijection
such that $n$ and $f(n)$ are equal.

At the base is the \emph{prime spectrum\/} $\pi_G$, consisting
of the set of prime divisors of the order of $G$ (or equivalently, 
the set of prime orders of elements of $G$). The \emph{spectrum\/}
$\omega_G$ is the set of all possible element orders of $G$. 

The \emph{prime\/} (or \emph{Gruenberg-Kegel\/}) graph $\Gamma_G$
is the simplicial graph, or $1$-dimensional simplicial complex,
whose vertices are the set $\pi_G$ and with vertices $p\not= q$
joined by an edge exactly when $pq\in\omega_G$. 
As with the clique complex, the graph $\Gamma_G$ is
an abstract graph $\Gamma_G^\circ$ whose vertices correspond to
(or can be labelled by) distinct primes. 
An isomorphism $f:\Gamma_G\ra\Gamma_H$ is a graph
isomorphism $\Gamma_G^\circ\ra\Gamma_H^\circ$, where for 
the vertex $p\in\pi_G$
the image vertex $f(p)$ is also labelled by $p$.
The graph $\Gamma_G$ was introduced by Gruenberg and Kegel,
motivated by the results of \cite{MR0374247}; its connectedness
was first systematically studied in \cite{MR0617092}.

The prime spectrum $\pi_G$ is determined by 
the Gruenberg-Kegel graph $\Gamma_G$
--- just read the primes off the vertices --- and $\Gamma_G$
is in turn determined by the spectrum $\omega_G$. 
The abstract graph $\Gamma_G^\circ$ is obviously determined by
the labelled graph $\Gamma_G$. 

The spectrum $\omega_G$ is 
determined by the clique complex $X_G$: we have $m\in\omega_G$
exactly when $m=p_0^{m_0}\ldots p_i^{m_i}$, and there
is an $i$-face in $X_G$ spanned by vertices 
labelled by elements of orders $p_0^{m_0},\ldots, p_i^{m_i}$.
The element of order $m$ is then the product of the elements
labelling the vertices.
The Gruenberg-Kegel graph
$\Gamma_G$ is thus also determined by $X_G$; in fact:

\begin{lemma}
  \label{lemma3}
  There is a surjective map of simplicial graphs
  $X_G^{(1)}\ra \Gamma_G$.
  %where $X_G^{(1)}$ is the $1$-skeleton of
  %the clique complex. 
  %that sends the vertex $g\in X_G$
  %to the vertex $p\in\Gamma_G$,
  %and the edge $g_0g_1\in X_G$ to the edge in $\Gamma_G$ 
  %connecting $p_0$ to $p_1$.
\end{lemma}

\begin{proof}
If $g_0g_1$ is an edge
of $X_G$ then the $g_i$ are commuting $p_i$-elements with $p_0\not=p_1$.
There are thus powers of $g_0,g_1$ that are elements of orders
$p_0,p_1$, hence there is an edge of $\Gamma_G$ connecting the
vertices $p_0,p_1$. The map $f:X_G^{(1)}\ra \Gamma_G$
given by $f(g)=p$, for a $p$-element $g$,
and $f(g_0g_1)=$ the edge connecting $p_0$ to $p_1$,
is then a map of graphs.
If there is an edge $e$ connecting
$p_0,p_1$ in $\Gamma_G$ then there are commuting elements $g_0,g_1$ in $G$
of orders $p_0$ and $p_1$, and thus an edge $g_0g_1$ in $X_G$
with $f(g_0g_1)=e$. 
\qed
\end{proof}

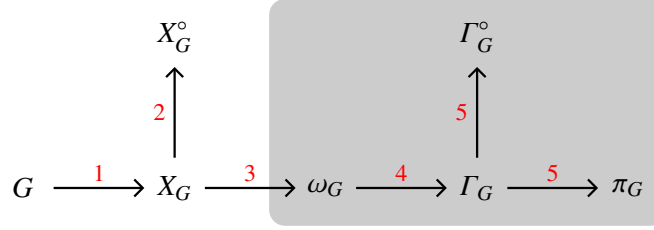
\begin{figure}
\begin{tikzpicture}
\draw [white] (0,0)--(\textwidth,4); % white line to determine figure
% height; leave
%
\begin{scope}[xshift=22.5mm,yshift=0mm]
\fill[fill=black!20, rounded corners=6pt] (4.25,0.5) rectangle (9.5,3.5);
\node[font=\normalsize] at(1,1){$G$};
\node[font=\normalsize] at(3,1){$X_G$};
\node[font=\normalsize] at(3,3){$X^\circ_G$};
\node[font=\normalsize] at(5,1){$\omega_G$};
\node[font=\normalsize] at(7,1){$\Gamma_G$};
\node[font=\normalsize] at(7,3){$\Gamma^\circ_G$};
\node[font=\normalsize] at(9,1){$\pi_G$};
\draw[line width=.3mm,-Straight Barb](1.4,1)--(2.6,1);
\draw[line width=.3mm,-Straight Barb](3,1.4)--(3,2.6);
\draw[line width=.3mm,-Straight Barb](3.4,1)--(4.6,1);
\draw[line width=.3mm,-Straight Barb](5.4,1)--(6.6,1);
\draw[line width=.3mm,-Straight Barb](7,1.4)--(7,2.6);
\draw[line width=.3mm,-Straight Barb](7.4,1)--(8.6,1);
\node[color=red] at(2,1.2){$1$};
\node[color=red] at(2.8,2){$2$};
\node[color=red] at(4,1.2){$3$};
\node[color=red] at(6,1.2){$4$};
\node[color=red] at(6.8,2){$5$};
\node[color=red] at(8,1.2){$5$};
\end{scope} 
%
%\draw[help lines,blue,line width=.6pt,step=1] (0,0) grid
%(\textwidth,4); % blue grid boxes; comment out when done
%
\end{tikzpicture}
\caption{A hierarchy of invariants for a finite group. There
are groups that are unrecognizable by the invariant
if and only if the invariant lies in the shaded area.}
\label{figureqe9t}
\end{figure}

This hierarchy of invariants is summarised in 
Figure \ref{figureqe9t}, with an arrow going from a stronger invariant
to a weaker one. 
The ultimate invariant is of course the group $G$ itself:
very discerning, but not very practical. All the other invariants in the figure
are incomplete, i.e., for each one, there are groups $G\not\cong H$
for which the invariant has the same value. 

In fact, it is possible to show that for each arrow, the invariant at the
starting end is strictly stronger than the invariant at the
finishing end. For example, the alternating group $A_{10}$ and 
the automorphism group $\text{Aut}(J_2)$ of Janko's second sporadic
group both have isomorphic underlying graph $\Gamma_G^\circ$
and the same prime spectrum $\pi_G$, but there is no isomorphism
between the Gruenberg-Kegel graphs: the $p$-labels on the
vertices are not compatible with any graph isomorphism.
(This is label $5$ on the edges of Figure \ref{figureqe9t}.)
The alternating
groups $A_5,A_6$ have isomorphic Gruenberg-Kegel graphs but
different spectra (label $4$). The symmetric groups $\Sym{5},\Sym{6}$
have the same spectra, but different clique complexes $X_G$
and $X_G^\circ$ (label $3$), because the order $|G|$ is determined
by $X_G^\circ$. See \cite{MR4506711}, and the references there, for all these
observations.

The clique complex of $G=\Z_2\times\Z_2$ consists of three isolated vertices,
all labelled by elements of order $2$; for the group $H=\Z_4$ we also have
three isolated vertices, with two labelled by elements of order $4$ and
one by an element of order $2$. Thus, $X_G^\circ\cong X_H^\circ$
but $X_G\not\cong X_H$, giving the label $2$.

Finally, for the clique complex itself we have:

\begin{example}
$\Z_4 = \{0,1,2,3\}$ has an automorphism of order $2$ swapping $1$ and 
$3$, so we can form a semidirect product $G = \Z_4\ltimes\Z_4$, where one copy of 
$\Z_4$ acts on the other via the natural surjective homomorphism 
$\Z_4\ra \Z_2=\text{Aut}(\Z_4)$.
Then $G$ has three elements of order $2$ and twelve elements of order $4$.
Let $H = Q_8\times \Z_2$ be the direct product of the quaternion group of order 
$8$ and the cyclic group of order $2$.
Then $H$ also has three elements of order $2$ and twelve elements of order $4$.

The clique complexes $X_G$ and $X_H$ both consist of $15$ isolated vertices,
with three labelled by elements of order $2$ and the remaining $12$ by
elements of order $4$. Thus $X_G\cong X_H$ but $G\not\cong H$,
and this gives label $1$ in Figure \ref{figureqe9t}.
This is not a low-dimensional phenomenon:
for $d\geq 1$, let $K_d=\Z_3\times\Z_5\times\cdots\times\Z_{p_d}$,
where $p_d$ is the $d$-th odd prime. Then
$X_{G\times K_d}\cong X_G*X_{K_d}\cong X_H*X_{K_d}
\cong X_{H\times K_d}$, where $G\times K_d$ and $H\times K_d$
are non-isomorphic groups of dimension $d$.
\end{example}

Let $\Omega_G$ be one of the invariants in Figure \ref{figureqe9t}.
In \cite{MR4506711} a group $G$ is said to be recognizable 
(or \emph{seen\/}) by $\Omega_G$, if for any $H$, an isomorphism
$\Omega_G\cong\Omega_H$ implies that $G\cong H$. Thus, $G$ is determined
up to isomorphism by $\Omega_G$. The group $G$ is \emph{almost
recognizable\/} by $\Omega_G$ if there are only finitely many 
non-isomorphic $H$ with $\Omega_G\cong\Omega_H$, and
\emph{unrecognizable\/} when there are infinitely many
non-isomorphic $H$ with $\Omega_G\cong\Omega_H$.

There is an extensive literature around these terms, for which
we refer the interested reader to \cite{MR4506711}, and
also for the following remarks.
Of immediate interest to us is the following result:
if $G$ contains a non-trivial
soluble normal subgroup, then $G$ is unrecognizable
by $\omega_G$, and hence by $\Gamma_G$ and $\pi_G$ too.
On the other hand:

\begin{proposition}
A finite group $G$ is almost recognizable by $X_G^\circ$,
and hence by $X_G$ too.
\end{proposition}

This is for the simple reason that $X_G^\circ$ determines
the order $|G|$ of $G$, and there are only finitely many isomorphism
classes of groups with a given order.
Updating Figure \ref{figureqe9t}, there are $G$ that are
unrecognizable by the invariant $\Omega_G$
if and only if $\Omega_G$ lies in the shaded area.

The situation becomes more interesting for the finite simple
groups: almost all finite simple groups are almost recognizable
by $\omega_G$. By this we mean that there is a finite set $S$
of simple groups, such that if $G$ is finite simple, and $G\not\in S$,
then $G$ is almost recognizable by $\omega_G$. On the other hand:

\begin{theorem}
\label{theorem:simple:groups}
A finite simple group $G$ is seen by the clique complex $X_G$.
\end{theorem}

The result follows from the following one: if $G$ is finite simple
and $H$ is a finite group, then $G\cong H$ if and only if 
$\omega_G=\omega_H$ and $|G|=|H|$; see \cite{MR2640961} and the 
references there. 
The
latter two conditions are implied by $X_G\cong X_H$.

%%%%%%%%%%%%%%%%%%%%%%%%%%%%%%%%%%%%%%%%%%%%%%%%%%%%%
%%%%%%%%%%%%%%%%%%%%%%%%%%%%%%%%%%%%%%%%%%%%%%%%%%%%%

\section{$X_G$ for certain families of groups}

We describe the clique complex $X_G$ for certain $G$, with
a particular focus on dimensions and mono-connectedness.
The Abelian groups are the easiest, but as they can be dealt
with uniformly among the nilpotent groups, we postpone
them to Section \ref{section:clique:complex:nilpotent}.

%%%%%%%%%%%%%%%%%%%%%%%%%%%%%%%%%%%%%%%%%%%%%%%%%%%%%

\subsection{The symmetric groups}
\label{section:clique:complex:symmetric}

It is possible to describe the facets and isolated vertices in the
clique complex $X_{\Sym{n}}$.

\begin{lemma}
  \label{lemma:symmetric:decomposition}
  The decomposition of Lemma \ref{lemma1} for $g\in\symn$ is obtained
  as follows:
  \begin{enumerate}[(i).]
    \item Decompose $g\in\symn$ into a product of disjoint
  cycles.
\item For each $k$-cycle in this decomposition,
  with $k=q_0q_1\ldots q_\ell$
  where $q_i=p_i^{m_i}$ for $p_i$ prime,
  we have the decomposition $g_0g_1\ldots g_\ell$,
  where $g_i$ is a product of
  $q_0q_1\ldots\hat{q}_i\ldots q_\ell$
  cycles of length $q_i$, with $\hat{q}_i=1$.
\end{enumerate}
\end{lemma}

We have done an example of this process already, on the
left hand side of Figure \ref{figure1}, where $g$ has 
cycle structure $(7,5^2,6,2^2)$ and the cycles with length
a prime power appear as themselves in the decomposition;
the $6$-cycle decomposes as a product of two $3$-cycles and
three $2$-cycles.

An element $g\in \symn$ has cycle structure
$(\mu_1,\mu_2,\ldots,\mu_t)
=(\lambda_1^{a_1},\ldots,\lambda_s^{a_s})\partition n$, a partition
of the integer $n$, where the $\lambda_i$ come about by collecting
together the $\mu_j$ that have the same size.

\begin{lemma}
\label{lemma:symmetric:facets}
  \begin{enumerate}[(i).]
  \item $\dim\Sym{n}+1$ is the largest $k$ for which the
    sum $2+3+\cdots+p_k$ of the first $k$ primes is less than or equal
    to $n$. In particular,
    $\dim\Sym{n}$ tends to infinity as
    $n$ tends to infinity.
\item Let $g\in \symn$ have cycle structure
  $(\lambda_1^{a_1},\ldots,\lambda_s^{a_s})$. Then
  $g$ is a vertex if and only if there is
  a prime $p$ such that each $\lambda_j$ is a power of $p$.
\item $g$ is a facet if and only if whenever a prime $p$
is $\leq$ one of the $a_i$, then $p$ also divides one of the 
$\lambda_j$.
\item $g$ is an isolated vertex if and only if there is
  a prime $p$ such that each $\lambda_j$ is a power of $p$, and
  $p$ is the only prime $\leq a_i$ for all $i$.
  \end{enumerate}
\end{lemma}

\begin{proof}
An element with cycle structure
$(p_k,\ldots,3,2,1^m)\partition n$ has order with the largest possible number of distinct
prime divisors amongst all elements of $\Sym{n}$.
%\hl{MB: I changed the preceding sentence, which previously seemed to be claiming some sort of uniqueness. But, e.g., for $n=9$ we have elements of cycle structure $(2,3)$, $(2,5)$ and $(3,5)$ all of same dimension.}
Lemma \ref{lemma2}(i) then
establishes (i). Also by
Lemma \ref{lemma2}(i)
the element $g\in \symn$
has dimension $\omega(g)-1$, where $\omega(g)$ is
the number of distinct primes dividing the lowest common
multiple of the $\{\lambda_1,\ldots,\lambda_s\}$. Part (ii) follows. 
If $g$ has cycle structure
$(\lambda_1^{a_1},\ldots,\lambda_s^{a_s})$ then $g$
has centralizer the product of wreath products:
$$
C_g\cong
\prod_i^s \Z_{\lambda_i} \wr \Sym{a_i},
$$
with $|C_g|=\prod \lambda_i^{a_i}a_i!$.
For a prime $p$ that divides this order to also divide the order
of $g$ we must therefore have the condition stated in part (iii),
and the result follows by Lemma \ref{lemma2}(ii).
Part (iv) follows from Lemma \ref{lemma2}(iii).
\qed
\end{proof}

It follows from the previous result that, in particular, $g\in \Sym{n}$ is an isolated vertex precisely when $g$
has cycle structure
  $(\lambda_1^{a_1},\ldots,\lambda_s^{a_s})$ and:
  \begin{enumerate}[({S}1).]
\item the $\lambda_i$ are distinct powers of $2$ with the
     $a_i=1$ or $2$; or,
\item the $\lambda_i$ are distinct powers of an odd prime $p\leq n$ with
     the $a_i=1$.
\end{enumerate}
For example, writing $n$ to base $2$, i.e. in the form
$n=n_s\cdot 2^s+\cdots +n_1\cdot 2+n_0$, where the $n_i=0,1$,
then a $g$ with cycle structure
$(\lambda_1,\ldots,\lambda_s)$ where the $\lambda_i$ are those $2^i$
for which $n_i=1$, will be an isolated vertex in $\symn$ for all
$n\geq 2$. The symmetric groups thus always contain isolated vertices.
On the other hand:

\begin{proposition}
  \label{proposition1}
   \begin{enumerate}[(i).]
\item $\symn$ is pure if and only if $n\leq 4$.
\item $\symn$ is mono-connected if and only if $n\not=5,6,7$.
%\item $\symn$ is connected if and only if \hl{$n\geq 15$ is odd and for
%  every odd prime $p\leq n$, the base $p$ expansion of $n$ 
%  does not have the form $(n_sn_{s-1}\ldots n_10)_p$, with the $n_i=0$ or $1$.}
  \end{enumerate}
\end{proposition}

\begin{proof}%[of Proposition \ref{proposition1}]
For $n\leq 4$, all elements of $\Sym{n}$ have prime power order,
so the clique complex just consists of isolated vertices
by Lemma \ref{lemma4}(i). (Alternatively,
$\dim\symn=0$ precisely when $n\leq 4$.) In any case, $\Sym{n}$
is pure and mono-connected.

There is an isolated vertex, hence a $0$-dimensional
facet, in $\symn$ for all $n\geq 5$.
To see the lack of purity it thus suffices to
exhibit a face of dimension $>0$. 
Since $n\geq 5$, the element $(1,2)(3,4,5)\in \Sym{n}$ will do.
%\hl{MB: I changed this preceding sentence. The old text is commented out in the proof, along with the reason I didn't believe it.}

%If $n$ is not prime then
%an $n$-cycle is a facet having dimension $>0$. If $n$ is prime
%then an $(n-1)$-cycle suffices.
%\hl{MB: don't believe this! If $n = 2^n+1$ is a Fermat prime, then an $(n-1)$-
%cycle is an isolated vertex.}

For $5\leq n\leq 7$ some easy calculations show that the clique complex
has both isolated vertices and multiple larger connected components. 
For example, when $n=7$, there is a single component containing the
$2$-, $3$-, $4$- and $5$-cycles and the double transpositions.
Fix $i\in\{1,2,\ldots,7\}$; then the elements of cycle types
$(2^3,1)$ and $(3^2,1)$, where the single fixed point is $i$, form
a connected component. There are thus another $7$ separate components;
also, the $7$-cycles are all isolated vertices. These $\symn$ are not
mono-connected. 

For $n\geq 8$ we claim that
$\Sym{n}$ is mono-connected with 
the core $X_\circ$ consisting
of those vertices not satisfying (S1) or (S2) above.

\begin{enumerate}[(i).]
    \item If $(a,b)$ and $(c,d)$ are not necessarily disjoint
    $2$-cycles then there
    are edges $(a,b)\,\text{---}\,(e,f,g)$ $\,\text{---}\,(c,d)$, 
    where $e,f,g\not\in\{a,b,c,d\}$. If 
    $(a,b,c)$ and $(d,e,f)$ are not necessarily disjoint $3$-cycles then 
    there are 
    edges $(a,b,c)\,\text{---}\,(g,h)\,\text{---}\,(d,e,f)$
    where $g,h\not\in\{a,b,c,d,e,f\}$.
    Thus,
    all $2$- and $3$-cycles are in $X_\circ$.
    \item A $2$-element that fixes three or more points is connected
    by an edge to a $3$-cycle in $X_\circ$, and a $p$-element ($p>2$)
    that fixes two or more points is connected by an edge to
    a $2$-cycle in $X_\circ$.
    \item If $g$ is a non-isolated vertex, not already dealt with, then
    either (a). $g$ is a $p$-element ($p>2$) that fixes $0$ or $1$ points
    and with at least one repeated cycle length, or (b). $g$ is a $2$-element
    that fixes $0,1$ or $2$ points and has a cycle length which is
    repeated at least $3$ times. 
    \item In case (a), suppose that $g$ contains at least two $p^e$-cycles
    for some $e$.
    Then $g$ is centralized by an element $h$ of cycle structure $2^{p^e}$,
    which in turn is centralized by a $k$ of shape $3^2$, i.e. there
    are edges $g\,\text{---}\,h\,\text{---}\,k\text{---}\,(x,y)\in X_\circ$,
    where, as $n\geq 8$, the points $x,y$ are chosen to 
    not be in the support of $k$.
    \item In case (b), suppose that $g$ contains at least three 
    $2^e$-cycles. Then $g$ is centralized by an $h$ of shape 
    $3^{2^e}$. If $e=0$ then $g$ fixes at least three points,
    a contradiction, hence $2^e\geq 2$ and $h$ in turn is centralized
    by a $k$ of shape $2^3$, and \emph{this\/} in turn by an $\ell$
    of shape $3^2$. This gives edges 
    $g\,\text{---}\,h\,\text{---}\,k\text{---}\,\ell\,\text{---}\,(x,y)\in X_\circ$,
    where $x,y$ are not in the support of $\ell$.
\end{enumerate}
This completes the proof of the claim.
\qed 
\end{proof}

%%%%%%%%%%%%%%%%%%%%%%%%%%%%%%%%%%%%%%%%%%%%%%%%%%%%%

\subsection{Nilpotent and soluble groups}
\label{section:clique:complex:nilpotent}

The nilpotent groups have some of the better-behaved clique
complexes.

If $G$ is nilpotent of order
$q_1q_2\ldots q_n=p_i^{m_1}p_2^{m_2}\ldots p_n^{m_n}$
with the $p_i$ distinct primes
(in particular,
if $G$ is Abelian with this order) then $G$ is isomorphic
to the direct product of its Sylow subgroups:
$G\cong G_1\times G_2\times\cdots\times G_n$ with $G_i$ having order
$q_i$; see, for example \cite{MR1357169}*{5.2.4}.
By Lemma \ref{lemma4}, the clique complex $X_G$
is the
$n$-fold join $X_{G_1} * X_{G_2}*\cdots * X_{G_n}$,
with $X_{G_i}$ a set of $q_i-1$ isolated vertices.

The clique complexes of nilpotent groups are thus members of
a more general family: let $a>0$ be an integer and 
write $X(a)$ for the $0$-dimensional 
complex consisting of $a$ isolated vertices. If 
$a_1,a_2,\ldots,a_n$ are integers
with $a_i>0$, then let:
  \begin{equation}
\label{equationqt8jpb4ymt}
    X(a_1,a_2,\ldots,a_n)= 
%    \bigast_{i=0}^n X(a_i)
%    {\mathlarger{\mathlarger{\mathlarger{*}}}}_{i=0}^n X(a_i),
X(a_1)*X(a_2)*\cdots *X(a_n),
\end{equation}
be the $n$-fold join of the spaces $X(a_i)$. The clique complex
$X_G$, for $G$ nilpotent, is then the complex $X(q_1-1,q_2-1,\ldots,q_n-1)$.
The ``$-1$"s are because $X_{G_i}$ is composed of 
$|G_i|-1$ isolated vertices.
An $i$-face of $X(a_1,a_2,\ldots,a_n)$
corresponds to an $n$-tuple $(x_1,x_2,\ldots,x_n)$ where each $x_j$ 
is an integer $0\leq x_j\leq a_j$ and with exactly $i+1$ of the 
$x_i\not= 0$.
For example, $(x_1,0,\ldots,0)$ and
$(0,x_2,0,\ldots,0)$ are vertices and
$(x_1,x_2,0,\ldots,0)$ is an edge that joins them.
The facets are the $n$-tuples with
$x_i\not= 0$ for all $i$.
In particular, $X(a_1,a_2,\ldots,a_n)$ is pure $(n-1)$-dimensional.
We include in this family the empty complex, which arises
when the set of $a$'s is empty.

If $n=1$ then $X(a_1)$ is disconnected with $a_1$ isolated vertices.
In particular $X(a_1)$ is mono-connected: the non-isolated vertices
are the empty complex, and the empty complex is connected.
If $n>1$ then $X(a_1,a_2,\ldots,a_n)$ is connected: the vertices
$(0,\ldots,0,x_j,0,\ldots,0)$ and $(0,\ldots,0,x_k,0,\ldots,0)$
for $j\not= k$,
are connected by the edge $(0,\ldots,0,x_j,0,\ldots,0,x_k,0,\ldots,0)$.

In any case:

\begin{proposition}
The nilpotent groups are mono-connected.
\end{proposition}

The complexes (\ref{equationqt8jpb4ymt}) also provide a description
of the $p$-restricted clique complexes of nilpotent groups. If $p$ is 
equal to $p_i$ for some $i$, then 
$X_G^p$ consists of those tuples
$(x_1,\ldots,x_{i-1},0,x_{i+1},\ldots,x_n)$
where there is no restriction on the $x_j$
for $j\not= i$.
In particular, for $G$ nilpotent as above:
$$
X_{G}^{p}\cong X(q_1-1,\ldots,\widehat{q_i-1},\ldots,q_n-1),
$$
with the hat denoting omission.

Unfortunately, our understanding for soluble groups
is very incomplete,
except in one simple case:
the dihedral group of symmetries of an $n$-gon,
for $n\geq 3$.
Although this is no longer nilpotent in general, it has a cyclic
subgroup $\Z_n$ of rotations of index $2$, and the remaining
elements are reflections of order $2$.
The reflections form a single conjugacy class of size $n$ when $n$ is odd,
and two classes of size $n/2$ when $n$ is even.
In either case, the centralizer of a reflection is a $2$-group, so
these are isolated vertices in the clique complex by
Lemma \ref{lemma2}(iii).
The remaining elements of the dihedral group lie in the cyclic subgroup,
and so the clique complex for the dihedral group consists
of the subcomplex $X_{\Z_n}$, together with a cloud of
isolated vertices corresponding to the reflections.
If $n$ is a prime power then $X_{\Z_n}\cong X(n)$, 
composed entirely of isolated vertices, otherwise
$n=q_1q_2\ldots q_m\,(m\geq 2)$ and 
$X_{\Z_n}\cong X(q_1-1,q_2-1,\ldots,q_m-1)$
is connected.

In any case, the dihedral groups are mono-connected for all $n\geq 3$.

%%%%%%%%%%%%%%%%%%%%%%%%%%%%%%%%%%%%%%%%%%%%%%%%%%%%%

\subsection{The alternating groups}
\label{section:clique:complex:alternating}

We now turn our attention to the simple groups.

The
picture for the alternating groups
is similar to that for the symmetric groups, and so we just state
the results and leave the justifications to the reader. 
The dimension of an element $g\in A_n$ is inherited
from its dimension in $\symn$, 
and if $p_0,p_1,\ldots,p_i$ are distinct \emph{odd\/}
primes then a $g\in A_n$ with cycle structure 
$(p_0,p_1,\ldots,p_i)$ will be an $i$-face in the clique complex,
once $n$ is sufficiently large. The dimensions $\dim A_n$
thus tend to infinity as $n$ does. 

An analysis of centralizers shows that the facets
in $A_n$ are:
\begin{enumerate}[(i).]
    \item The facets of $\symn$ that also happen to be elements of 
    $A_n$, i.e. the even permutations that satisfy 
    Lemma \ref{lemma:symmetric:facets}(iii).
    \item The $g\in A_n$ with cycle structure
  $(\lambda_1^{a_1},\ldots,\lambda_s^{a_s})$ with all the $\lambda_j$
  odd, exactly one of the $a_i=2$, and the rest equal to $1$.
  \item The $g\in A_n$ with all the $\lambda_j$
  odd and with at least one divisible by $3$, 
  one of the $a_i=3$, and the rest equal to $1$.
\end{enumerate}
The elements in cases (ii) and (iii) are not facets in $\symn$. 

The isolated vertices are thus those facets of $\symn$ which are vertices
--- Lemma \ref {lemma:symmetric:facets}(iv) or (S1) and (S2) above ---
and also happen to lie in $A_n$, together with the following additional ones:
\begin{enumerate}[({A}1).]
\item The $g$ with cycle structure
  $(\lambda_1^{a_1},\ldots,\lambda_s^{a_s})$, where for a fixed prime
  $p>2$, each $\lambda_j$ is a distinct power of $p$, exactly one of the
  $a_i=2$, and the rest are equal to $1$.
\item The $g$ with each $\lambda_j$ a distinct power of $3$,
one of the $a_i=3$ and the rest equal to $1$.
\end{enumerate}

An analogous (albeit slightly more involved) calculation to that 
found in the proof of Proposition \ref{proposition1}
shows that any two elements of $A_n$, subject to certain conditions
on $n$, that are not of the types
(S1), (S2), (A1) or (A2), lie in the same connected
component. 
In particular:

\begin{proposition}
  \label{proposition:alternating:monoconnected}
   \begin{enumerate}[(i).]
\item $A_n$ is pure if and only if $n\leq 6$.
\item $A_n$ is mono-connected if and only if $n\not=7,8,9$.
  \end{enumerate}
\end{proposition}

%%%%%%%%%%%%%%%%%%%%%%%%%%%%%%%%%%%%%%%%%%%%%%%%%%%%%

\subsection{Finite Simple Groups of Lie Type}
\label{section:clique:complex:classical}

Let $G = G(q)$ be a finite simple group of Lie type over a field
$\F_q$ of order $q = p^e$ for some prime $p$.
Then the $\ell$-elements for primes $\ell$ dividing $|G|$ have two
flavours: either $\ell = p$, in which case we are considering a
\emph{unipotent} element, or $\ell\neq p$.

A large supply of non-unipotent elements are those contained 
in the maximal tori in $G$. 
These can be described,
see 
\cite{springersteinberg}*{CH. ii, \S1}. 
For example, when $G = \mathrm{PSL}_n(q)$ the conjugacy classes of
maximal tori correspond to partitions of $n$ as follows:
given a partition $(d_1,\ldots,d_k)$ of $n$ with $d_1\geq \cdots \geq d_k>0$,
we can realise the multiplicative group of $\F_{q^{d_i}}$
as a group of $d_i\times d_i$ matrices with entries in $\F_q$,
and then we can embed these groups \emph{block diagonally} in
$\mathrm{GL}_n(q)$ to get a corresponding maximal torus of that group.
Intersecting with $\mathrm{SL}_n(q)$ and passing to the quotient $\mathrm{PSL}_n(q)$ gives the maximal tori in these groups.

Since these tori are commutative, generically they provide 
examples of faces of
large dimension in the clique complex. 
In particular, we see that as $q$ or $n$ grows, the dimension of
facets of the clique complex can grow arbitrarily large.
(An analogous argument will hold for finite simple groups of any Dynkin type.)

The connectedness properties of the clique complex for Lie type groups depend both on the rank and the field.
A systematic analysis is tangential to the main thrust of the paper, but we can relatively easily settle the rank $1$ case.
In order to do this, we need information about the conjugacy classes in $G = \mathrm{PSL}_2(q)$, which is also useful for results later in the paper.

We first deal with even $q$, so suppose $q = 2^m$. 
Then in fact $G=\mathrm{SL}_2(q)$, and the conjugacy classes in $G$ are described in Table \ref{table:PSL:q:2}, which is deduced from \cite{MR0347959}*{Theorem 38.2}.
\begin{table}[t]
    \centering
  \begin{tabular}{cccc}\hline
    $\mathrm{PSL}_2(2^m)$&$D$&$A^\ell\,(1\leq\ell\leq\frac{q-2}{2})$&$B^\ell\,(1\leq\ell\leq\frac{q}{2})$\vrule width 0mm height 4 mm depth
  2mm\\\hline
  $g$&$2$&$[q-1,\ell]$&$[q+1,\ell]$\vrule width 0mm height 4 mm depth
  0mm\\
    $g^G$&$(q-1)(q+1)$&$q(q+1)$&$q(q-1)$\vrule width 0mm height 5 mm depth
  0mm\\
    $C_g$&$q$&$q-1$&$q+1$\vrule width 0mm height 5 mm depth
  2mm\\
    \hline
  \end{tabular}
\caption{Conjugacy classes of $\mathrm{PSL}_2(q)$ for $q=2^m\,(m\geq 2)$}
    \label{table:PSL:q:2}
\end{table}
The first row gives 
the order of the element $g$,
the second row
the size of the conjugacy class
$g^G$, and the third row the
size of the centralizers $C_g$  --- as $\mathrm{PSL}_2(q)$
has order $q(q-1)(q+1)$ in this case. 
The symbol $[x,y]$, for $x,y$ positive integers, denotes the quantity 
$x/\kern-2pt\gcd(x,y)$.
The classes in Table \ref{table:PSL:q:2} are represented by powers of the following elements:
\[
D = \left(\begin{array}{cc} 1 & 1 \\ 0 & 1 \end{array}\right), \quad
A = \left(\begin{array}{cc} \alpha & 0 \\ 0 & \alpha^{-1}\end{array}\right), \quad
B \text{ of order $q+1$}.
\]
Here $\alpha\in \F_q$ is a generator for the multiplicative group of nonzero elements, so $A$ generates a torus of order $q-1$, which is exactly the centralizer of $A$ (and all its nontrivial powers).

The element $B$ generates a maximal torus of $G$ constructed as described above:
viewing $\F_{q^2}$ as a two-dimensional $\F_q$-vector space, the elements of its multiplicative group can be represented as $2\times 2$ matrices with entries in 
$\F_q$.
If $B'$ is the matrix representing a generator for this group, 
then $B = (B')^{q-1}$.
The centralizer of $B$ and each of its powers is the group generated by $B$.

Now we turn to odd $q$, so suppose $q = p^m$ with $p$ an odd prime.
This time the projection map $\mathrm{SL}_2(q)\to G$ has kernel of order $2$; we denote this map by $(a_{ij})\mapsto [a_{ij}]$.
The conjugacy class information about $G$ is recorded in Table \ref{table:PSL:q:odd}; it can be deduced from \cite{MR0347959}*{Theorem 38.1}.
\begin{table}[t]
    \centering
\begin{tabular}{cccccc}\hline
    $\mathrm{PSL}_2(q)$&$D$&$H_1$&$H_2$&$A^\ell\,(1\leq\ell\leq m_1)$
    &$B^\ell\,(1\leq\ell\leq m_2)$\vrule width 0mm height 4 mm depth
  2mm\\\hline
  $g$&$2$&$q$&$q$&$[\frac{q-1}{2},\ell]$
  &$[\frac{q+1}{2},\ell]$\vrule width 0mm height 4 mm depth
  0mm\\
    $g^G$&$\frac{q(q\pm 1)}{2}$&$\frac{(q-1)(q+1)}{2}$&$\frac{(q-1)(q+1)}{2}$&$q(q+1)$
    &$q(q-1)$\vrule width 0mm height 5 mm depth
  0mm\\
    $C_g$&$q\mp 1$&$q$&$
    q$&$\frac{q-1}{2}$&$\frac{q+1}{2}$\vrule width 0mm height 5 mm depth
  2mm\\
    \hline
  \end{tabular}
      \caption{Conjugacy classes of $\mathrm{PSL}_2(q)$ for $q$
      a power of an odd prime. The numbers
      $m_1,m_2$ are explained in the text.}
    \label{table:PSL:q:odd}
\end{table}
Again, we list the order of the element $g$ in the first row, the sizes of the 
classes in the second row, and the sizes of the centralizers in the third row; as 
before, $[x,y]:=x/\kern-2pt\gcd(x,y)$.
This time, we have two classes of unipotent elements, with representatives 
labelled $H_1$ and $H_2$. 
The elements $A$ and $B$ in the table are analogous to those in the even 
characteristic case: $A$ is the image in $G$ of a diagonal element in 
$\mathrm{SL}_2(q)$ of order $q-1$, and $B$ is the image in $G$ of  an element of 
order $q+1$.
An added complication in odd characteristic arises because precisely one of the 
powers of one of the elements $A$ or $B$ gives an element of order $2$, labelled 
$D$ in the table.
Whether it is $A$ or $B$ giving rise to $D$ depends on whether $q$ is congruent 
to $1$ or $3$ modulo $4$.
If $q\equiv 1\kern-2mm\mod 4$, then $\frac{q-1}{2}$ is even, and 
$D = A^{\frac{q-1}{4}}$ is the element of order $2$;
in this case, we let $m_1 = \frac{q-5}{4}$ and $m_2 = \frac{q-1}{4}$.
Similarly, if $q\equiv 3\kern-2mm\mod 4$, then $\frac{q+1}{2}$ is even, and 
$D = B^{\frac{q+1}{4}}$ is the element of order $2$;
in this case, we let $m_1 = m_2 = \frac{q-3}{4}$.

The centralizer of each power of $A$ (respectively, $B$) described in the table 
is the cyclic group generated by $A$ (respectively, $B$).
To describe the centralizer of $D$, let:
\[
w = \left[\begin{array}{rc} 0 & 1\\ -1 & 0\end{array}\right],
\]
be a representative of the nontrivial element of the Weyl group of $G$. Then $w$ 
acts on the cyclic subgroup generated by $A$ by inversion, and the same for the 
cyclic subgroup generated by $B$.
In the case $D$ is a power of $A$, the centralizer of $D$ is therefore the 
dihedral group of order $q-1$ generated by $A$ and $w$;
in the case $D$ is a power of $B$ the centralizer is the dihedral group of order 
$q+1$ generated by $B$ and $w$.

\begin{proposition}\label{prop:PSL2_cliquecomplex}
Let $q$ be a prime power and let $G = \mathrm{PSL}_2(q)$.
Then $X_G$ is mono-connected if and only if $q = 2, 3, 4, 5, 7, 8$ or $9$.
\end{proposition}

\begin{proof}
We first deal with even $q$, so suppose $q = 2^m$, and refer to Table 
\ref{table:PSL:q:2}  and its description above. 
The element $D$ has order $2$ and centralizer of order $q = 2^m$, so $D$ and its 
conjugates give isolated vertices in the clique complex.
Each element $A^i$ has centralizer of order $q-1$, which is just the subgroup of 
diagonal matrices. 
Each element $B^j$ has centralizer of order $q+1$, which is just the subgroup 
generated by $B$.
Since $q-1$ and $q+1$ are coprime and odd, there are no edges in the clique 
complex between elements of the form $A^i$ and $B^j$. 

The $\frac{q-2}{2}$ conjugacy classes of elements $A^{\ell}$ together account 
for $\frac{q(q+1)(q-2)}{2}$ elements of $G$.
A quick matrix calculation shows that the only conjugate of $A$ which is still 
diagonal is $A^{-1}$, so the $q(q+1)$ elements in the conjugacy class of $A$ 
give rise to $\frac{q(q+1)}{2}$ distinct centralizers of order $q-1$.
Each of these centralizers has $q-2$ non-identity elements and together cover 
all the conjugates of all powers of $A$, since the centralizer of $A$ itself is 
just the powers of $A$.
We conclude that these $\frac{q(q+1)}{2}$ centralizers must all have trivial 
pairwise intersection.
The upshot is that when $q-1$ is not a prime power, we have $\frac{q(q+1)}{2}$ 
distinct connected components of size greater than $1$ in the clique complex; 
since $\frac{q(q+1)}{2} >1$ in all cases, we see that $X_G$ is not mono-connected in these cases.

A similar argument to the previous paragraph works for the conjugates of 
elements $B^\ell$.
Hence if $q+1$ is not a prime power and $\frac{q(q-1)}{2}>1$, then $X_G$ is not 
mono-connected; this latter is fine as long as $q\neq 2$.
Now Catalan's conjecture --- proved
by Mih\u{a}ilescu in \cite{MR2076124}
--- implies that it can only happen that $2^m-1$ and $2^m+1$ are both positive 
powers of primes when $q=4$ or $q=8$.  
In these cases the clique complex just consists of isolated vertices and so is 
vacuously mono-connected.

Now we turn to odd $q$, so suppose $q = p^m$ with $p$ an odd prime and refer to 
Table \ref{table:PSL:q:odd} and its description above.
The unipotent elements represented by $H_1$ and $H_2$ still give isolated 
vertices in the clique complex.
The powers of $A$ and $B$ in the table still just have centralizers powers of 
$A$ or $B$, respectively.
In either case of the element $D$, the centralizer of $D$ does not provide any 
new edges in the clique complex, since the ``extra'' elements centralizing $D$ 
are just the reflections in the corresponding dihedral group, and $D$ itself has 
order $2$ (see Section \ref{section:clique:complex:nilpotent} above).
The remaining analysis is exactly the same as in the even characteristic case.
So we just need to ensure that at least one of  $\frac{q-1}{2}$ or 
$\frac{q+1}{2}$ is not a prime power. 
Since these are consecutive integers, one is even and one is odd; hence if they 
are both prime powers, one of them is a power of two.
If $\frac{q-1}{2} = 2^\ell$ for some $\ell$, then $q = 2^{\ell +1} + 1$.
By Catalan's conjecture again, we either have $q = 9$ or $q$ is a Fermat prime. 
We see that we must rule out the cases $\ell = 0$, $1$ and $2$, which are 
$q = 3$, $5$ and $9$.
Beyond these cases, for $q$ to be a prime, $\ell+1$ must itself be a power of 
$2$, so $\ell = 2^s - 1$, and then $\frac{q+1}{2} = 2^{\ell}+1 = 2^{2^s-1} +1$ 
is not a prime.
If $\frac{q+1}{2} = 2^\ell$ for some $\ell$, then $q = 2^{\ell +1} - 1$.
For this to be a nontrivial prime power, we have the small cases $\ell = 1$ and 
$\ell = 2$, which correspond to $q=3$ and $q=7$, or $q$ is a Mersenne prime, 
with $2<\ell = 2t$  even.
But then $\frac{q-1}{2} = 2^{\ell} - 1 = 2^{2t} - 1$ is composite.  

So we see that the only cases where all the vertices are isolated are as claimed 
in the statement. In these exceptional cases the clique complex for $G$ is a 
cloud of isolated vertices, so is trivially mono-connected.
In all other cases, it is not mono-connected.
\qed
\end{proof}

Note that we have already seen most of the ruled out cases of small $q$ in the proposition in different incarnations:
$\mathrm{PSL}_2(2) \cong \Sym{3}$,
$\mathrm{PSL}_2(3) \cong A_4$,
$\mathrm{PSL}_2(4)\cong \mathrm{PSL}_2(5) \cong A_5$,
$\mathrm{PSL}_2(9) \cong A_6$.

%We therefore pose the following \hl{question}:  
%Can you give a
%list of the simple groups which are not mono-connected?  
%\hl{Move whatever this turns out to be into Section
%2.6}

%%%%%%%%%%%%%%%%%%%%%%%%%%%%%%%%%%%%%%%%%%%%%%%%%%%%%

\subsection{The sporadic groups}
\label{section:clique:complex:sporadic}

The sporadic groups, despite the formidable size of some of them,
have quite small dimensions:

\begin{lemma}
  The $1$-dimensional sporadic simple groups are:
  $$
  M_{11},M_{12},M_{22},M_{23},M_{24},J_1,J_2,J_3,O^\prime N,
  Ru,HS,S\kern-1pt uz,
  $$
  while the remaining groups:
  $$
  J_4, Co_1,Co_2,Co_3,McL,He,HN,Th,Ly,Fi_{22},Fi_{23},Fi^\prime_{24},
  B,M,
  $$
  are $2$-dimensional.
\end{lemma}

This information can be read-off the $\ams{ATLAS}$
\cite{MR0827219}
entries
for these groups, where the conjugacy classes are listed and labelled
according to the orders of elements. For example, the Mathieu group
$M_{24}$ is $1$-dimensional, with vertices the conjugacy classes:
$$
2A,2B,3A,3B,4A,4B,4C,5A,7A,7B,8A,11A,23A,23B,
$$
(where $nA,nB,\ldots$ indicates the conjugacy classes
of elements of order $n$) and edges:
$$
6A,6B,10A,12A,12B,14A,14B,15A,15B,21A,21B.
$$
We can draw the schematic $\X_{M_{24}}$ for $M_{24}$ using information
about the orders of elements and the orders of their centralizers. 
The elements in $4B, 8A, 11A, 23A$ and $23B$ 
have centralizers with orders $2^7,2^4,11,23$ and $23$, and so
these are isolated vertices by Lemma \ref{lemma2}(iii).

The remaining analysis uses {\sc Gap} 
and the standard generators $a,b\in\Sym{24}$,
where $a$ lies in class $2B$ and $b$ in $3A$. 
For each edge conjugacy class above, decompose
a representative $g=g_0g_1$ and draw an edge
connecting the class of $g_0$ and $g_1$.
The resulting schematic $\X_{M_{24}}$ is shown 
in Figure \ref{figure2}.
%Consider for example
%the class
%$2B$, which contains the representative $g=(ababbababb)^3$ that has
%centralizer $C_{2B}$ of size $2^9\cdot 3\cdot 5$. This centralizer contains
%elements of order $5$, hence there is an edge corresponding to the
%class $10A$ joining the vertices $2B$ and $5A$.
%%There are $320$ elements of order $3$ in $C_{2B}$, all of which
%have centralizers of order $504$, hence lie in the class $3B$.
%This gives an edge $6B$ \hl{check not $6A$} joining
%$2B$ and $3B$.
%
%We claim that there are no other edges of $\X_{M_{24}}$ incident
%with $2B$. We can immediately rule out the possibility of
%edges to $2A,4A$ and $4C$. As all the elements of $3A$ have centralizer
%order $1080$, there is no edge to $3A$. The classes $7A,7B$ have centralizers
%containing elements of order $2$, but these elements all lie in $2A$ (again
%by looking at the orders of \emph{their\/} centralizers).

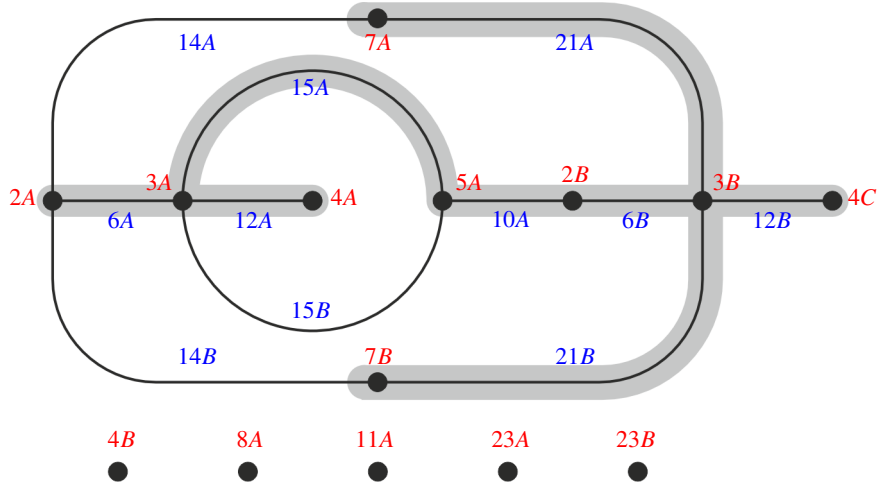
\begin{figure}
\begin{tikzpicture}
\draw [white] (0,0)--(\textwidth,1.75); % white line to determine figure
                               % height; leave
%
\node at (7.25,3.5) {\includegraphics[scale=1]{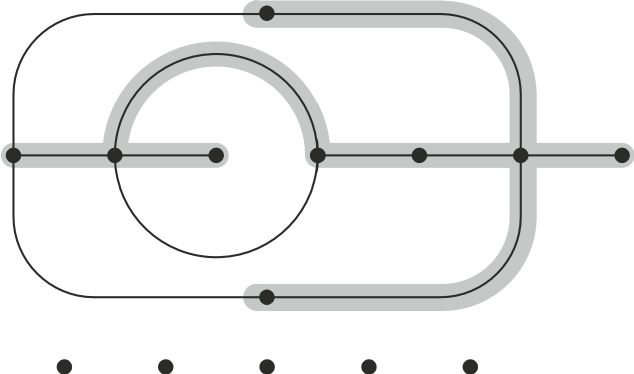}};
\node[color=red] at (1.7,4.1){$2A$};
\node[color=red] at (12.8,4.1){$4C$};
\node[color=red] at (5.95,4.1){$4A$};
\node[color=red] at (6.4,6.15){$7A$};
\node[color=red] at (6.4,2){$7B$};
\node[color=red] at (3,0.9){$4B$};
\node[color=red] at (4.7,0.9){$8A$};
\node[color=red] at (6.35,0.9){$11A$};
\node[color=red] at (8.15,0.9){$23A$};
\node[color=red] at (9.8,0.9){$23B$};
\node[color=red] at (7.6,4.3){$5A$};
\node[color=red] at (9,4.4){$2B$};
\node[color=red] at (11,4.3){$3B$};
\node[color=red] at (3.5,4.3){$3A$};
\node[color=blue] at (5.5,5.55){$15A$};
\node[color=blue] at (5.5,2.6){$15B$};
\node[color=blue] at (4,6.15){$14A$};
\node[color=blue] at (4,2){$14B$};
\node[color=blue] at (9,6.15){$21A$};
\node[color=blue] at (9,2){$21B$};
\node[color=blue] at (3,3.8){$6A$};
\node[color=blue] at (4.75,3.8){$12A$};
\node[color=blue] at (8.15,3.8){$10A$};
\node[color=blue] at (9.8,3.8){$6B$};
\node[color=blue] at (11.6,3.8){$12B$};
%
%\draw[help lines,blue,line width=.6pt,step=1] (0,0) grid
%(\textwidth,7); % blue grid boxes; comment out when done
%
\end{tikzpicture}
\caption{The schematic $\X_{M_{24}}$ for the Mathieu group $M_{24}$
and spanning tree $\goth{T}$.}
\label{figure2}
\end{figure}

%This completes the picture at the vertex $2B$ and the others are similar.
%The schematic $\X_{M_{24}}$ is shown in Figure \ref{figure2}.

%We now apply Lemma \ref{lemma5} with $\goth{C}$ the single connected
%component of $\X_{M_{24}}$ and $K_1=2A,K_2=3A$ and \hl{$K_3=7?$},
%to give that the subcomplex of $X_{M_{24}}$ spanned by the vertices
%of $\goth{C}$ form a connected component. The clique complex
%$X_{M_{24}}$ is thus mono-connected. \hl{check}

Now let $\goth{T}$ be the spanning tree for $\X_{M_{24}}$
shown in Figure \ref{figure2}. This induces a tree $T$ in 
the clique complex $X_{M_{24}}$, isomorphic to $\goth{T}$,
and having vertices in the indicated conjugacy classes. 
The vertices $g_1,g_2,g_3$ that are in the conjugacy classes
$2B,3B$ and $4C$ have centralizers that generate $M_{24}$,
and so $X_{M_{24}}$ is mono-connected by 
Lemma \ref{lemma5}. In fact:

\begin{proposition}
  \label{proposition:sporadic:monoconnected}
The Mathieu groups $M_{11},M_{23}$ and $M_{24}$ are mono-connected.
\end{proposition}

The analysis for all three is identical: the schematics are 
mono-connected, and the vertices of a spanning tree have centralizers
that generate. The Mathieu group $M_{12}$ is discussed in the next section. 
The non-isolated vertices in the schematic for $M_{22}$
form a single component, but the centralizers do not generate,
giving instead a subgroup of index $77$. The non-isolated vertices
of $X_{M_{22}}$ therefore have $1,7,11$ or $77$ components.

%%%%%%%%%%%%%%%%%%%%%%%%%%%%%%%%%%%%%%%%%%%%%%%%%%%%%

\subsection{Some more non mono-connected groups}
\label{section:clique:complex:twocomponents}

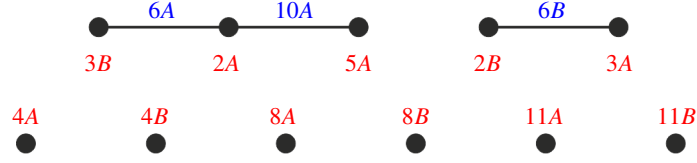
\begin{figure}
\begin{tikzpicture}
\draw [white] (0,0)--(\textwidth,3); % white line to determine figure
                               % height; leave
%
\node at (7.25,1.5) {\includegraphics[scale=1]{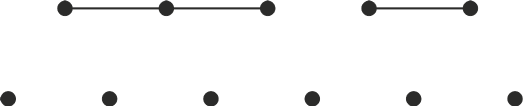}};
\node[color=red] at (5.6,1.8){$2A$};
\node[color=red] at (3.9,1.8){$3B$};
\node[color=red] at (7.35,1.8){$5A$};
\node[color=red] at (9.05,1.8){$2B$};
\node[color=red] at (10.8,1.8){$3A$};
\node[color=red] at (2.95,1.1){$4A$};
\node[color=red] at (4.65,1.1){$4B$};
\node[color=red] at (6.35,1.1){$8A$};
\node[color=red] at (8.1,1.1){$8B$};
\node[color=red] at (9.8,1.1){$11A$};
\node[color=red] at (11.55,1.1){$11B$};
\node[color=blue] at (4.75,2.5){$6A$};
\node[color=blue] at (6.5,2.5){$10A$};
\node[color=blue] at (9.9,2.5){$6B$};
%
%\draw[help lines,blue,line width=.6pt,step=1] (0,0) grid
%(\textwidth,3); % blue grid boxes; comment out when done
%
\end{tikzpicture}
\caption{The schematic $\X_{M_{12}}$ for the Mathieu group $M_{12}$.}
\label{figure2a}
\end{figure}

It is not hard to find examples of groups where the non-isolated 
vertices of the clique complex have precisely two connected components.
The Mathieu group $M_{12}$ has schematic as shown in 
Figure \ref{figure2a}; this is drawn in the same way as the schematics
for $M_{11},M_{23}$ and $M_{24}$ from the last section. One can find 
representatives $g_1,g_2,g_3$ for the conjugacy classes
$2A,3B,5A$ such that the $g_i$ span a connected tree in $X_{M_{12}}$
and the centralizers $C_{g_1},C_{g_2},C_{g_3}$ generate $M_{12}$.
Similarly, there are representatives $g_4,g_5$ for $2B,3A$ that span
an edge and whose centralizers generate. The non-isolated vertices 
of $X_{M_{12}}$ thus have two connected components with
$2^2\cdot 3^2\cdot 11+2^4\cdot 3\cdot 5\cdot 11+2^5\cdot 3^3\cdot 11$ 
and $3^2\cdot 5\cdot 11+2^5\cdot5\cdot 11$ vertices.

\vspace{1em}

One can also construct an infinite family of groups $G_p$,
for $p>3$ prime, with no isolated vertices and two connected
components, both of which have a number of vertices
that grows exponentially with $p$. 
Let $G_p=S_3\wr \Z_p$, the wreath product, with base
group $S_3\times \cdots \times S_3$ ($p$ times) 
and the cyclic group $\Z_p$ on top.
Then 
$$
|G_p| = 2^p \times 3^p \times p,
$$
so we have three primes to consider: $2$, $3$ and $p$.
Let $h$ be a generator of the cyclic group $\Z_p$ 
(written multiplicatively) and write a
general $g\in G_p$ as 
$(\sigma_1,\ldots,\sigma_p,h^i)$, with each $\sigma_j\in S_3$ and
$0\leq i\leq p-1$.
We single out three elements of $G_p$:
let $\sigma = (1,2)$ and $\tau = (1,2,3) \in S_3$, and set
$$
x = (1,\ldots,1,h); \quad s = (\sigma,\ldots,\sigma,1);
\quad t = (\tau,\ldots,\tau,1).
$$
Since the subgroup generated by $x$ has order $p$, it is a Sylow
$p$-subgroup, and so every $p$-element of $G_p$ has order $p$ and is
conjugate to some power of $x$.
Since the base group is normal in $G_p$, every $2$-element and every
$3$-element of $G_p$ lies in the base group.

The elements $s$ and $t$ generate a diagonal copy $D$ of $S_3$ in the
base group, and the centralizer of $x$ is the subgroup $D\times \Z_p$.
There are therefore $6^{p-1}(p-1)$ elements of order $p$.
The centralizer of $s$ is the subgroup $\langle \sigma \rangle \wr \Z_p$
of order $2^p\times p$, and $s$ has $3^p$ conjugates.
Similarly, the centralizer of $t$ is the subgroup
$\langle \tau\rangle \wr \Z_p$ of order $3^p \times p$ and $t$
has $2^p$ conjugates.
Some straightforward calculations show that all these elements
lie in a single component of the clique complex of $G_p$,
containing $6^{p-1}(p-1) + 3^p + 2^p$ vertices.

We are left with the $2$- and $3$-elements of $G_p$ which are not
conjugate to $s$ or $t$.
These elements all lie in the base group, and have the form
$(\sigma_1,\ldots,\sigma_p,1)$ with at least one $\sigma_i = 1$.
There are $4^p-3^p-1$ such elements of order $2$, and $3^p-2^p-1$
such elements of order $3$.
Since such an element commutes with at least one of the factors of
the base group, and since all the factors of the base group
commute with each other, we see quickly that they all lie in the
same connected component of the clique complex of $G_p$,
which this time contains $4^p-2^p - 2$ vertices.

So, here we have precisely two connected components,
both of which grow exponentially with $p$.
Essentially the same argument works for the groups $D_q\wr \Z_p$,
where $D_q$ is the dihedral group of order $2q$ and $q$ is an
odd prime different from $p$.

%%%%%%%%%%%%%%%%%%%%%%%%%%%%%%%%%%%%%%%%%%%%%%%%%%%%%
%%%%%%%%%%%%%%%%%%%%%%%%%%%%%%%%%%%%%%%%%%%%%%%%%%%%%

\section{(Combinatorial) sheaf homology}
\label{sheaves}

An early objective in the study of the Gruenberg-Kegel
graph $\Gamma_G$ was an understanding of its connectedness, and
the previous section devoted some energy to this question for
the clique complex $X_G$. The connectedness of a simplicial
complex $X$ is measured by the ordinary homology $H_0(X)$
in degree $0$, and so it is natural to study the
homology $H_*(X)$ more generally. 

In fact, we do more.
We equip the clique complex with a combinatorial sheaf
and study the resulting sheaf homology. 
While the ordinary homology
describes $X$ to some extent,
the sheaf homology often reveals more. 
%(I could give some examples here, but it is
%hard to do so without referencing my own work
%\cites{MR4492502,MR4401823,MR3276847},
%and then I look like a wanker. Wank away!)
In Section 
\ref{sheaves:fundamentals} we recall the
fundamentals of sheaf homology; standard references are
\cite{MR0210125}*{Appendix II} and \cite{MR2455920}.
Section 
\ref{sheaves:order:sheaves} describes the order sheaf on 
a clique complex and in Section \ref{sheaves:general:result}
we compute the sheaf homology of an arbitrary simplicial
complex with coefficients in an order sheaf. It turns
out that the result is
determined by the ordinary homology of the links 
of vertices. This is all interpreted for
clique complexes in Section \ref{sheaves:clique:complexes}.
One possibly surprising upshot is that the sheaf 
homology of $X_G$ can be computed with only partial
knowledge of the ordinary homology.

%%%%%%%%%%%%%%%%%%%%%%%%%%%%%%%%%%%%%%%%%%%%%%%%%%%%%

\subsection{Combinatorial sheaves and their homology}
\label{sheaves:fundamentals}

Let $R$ be a commutative ring with $1$ and $P$ a poset,
considered as a category with objects the elements of $P$
and a unique morphism $x\ra y$ if and only if $x\leq y$.
A \emph{sheaf\/} on $P$ is a \emph{contra\/}variant
functor $F:P\rightarrow\rmod$, the category of left $R$-modules.
If $x\leq y$ in $P$ then the $R$-module morphism
$F_x^y:F(y)\ra F(x)$ is called a
\emph{structure map\/} of the sheaf; 
when $x\leq y\leq z$ in $P$ the structure
maps satisfy $F_x^yF_y^z=F_x^z$, by functoriality.

A sheaf on a simplicial complex $X$ is a sheaf
on the poset $P_X$ of faces (under inclusion). 
Thus, a face $\ss$ is assigned an $R$-module
$F(\ss)$ and for $\tau\subseteq\ss$ there is a structure map
$F_{\tau}^{\ss}:F(\ss)\ra F(\tau)$. A triple of faces
$\tau\subseteq\ss\subseteq\mu$ gives 
structure maps that satisfy $F_\tau^\ss F_\ss^\mu=F_\tau^\mu$.

A morphism $\kappa:F\ra G$ of sheaves (either on a poset or a
simplicial complex) is a natural transformation of functors, i.e.
an $R$-module morphism $\kappa_\ss:F(\ss)\ra G(\ss)$ for each 
face $\ss$,
such that for any $\tau\subseteq
\ss$ we have
$G_\tau^\ss\kappa_\ss=\kappa_\tau F_\tau^\ss$. If each $\kappa_\ss$ is
an isomorphism then $\kappa$ is a sheaf isomorphism. 

Sheaves share many of the properties and
constructions of $R$-modules, and usually these
constructions can be performed ``locally". 
For example a sheaf morphism $\kappa:F\ra G$ has kernel
the sheaf whose value at a face $\ss$ is the kernel
of $\kappa_\ss$; similarly for images. A sequence
of sheaf morphisms $\cdots\ra F\ra G\ra\cdots$
is exact when the associated sequence
of module morphisms $\cdots\ra F_\ss\ra G_\ss\ra\cdots$
is exact for each face $\ss$.
The direct sum $\bigoplus_\lambda F_\lambda$
of sheaves has value at each $\ss$
the direct sum $\bigoplus_\lambda F_\lambda(\ss)$, and structure
maps the sums of the structure maps of the $F_\lambda$.

The colimit $\varinjlim F$
is the $R$-module obtained
by taking the quotient of $\bigoplus_{\ss\in X}F(\ss)$
by the submodule generated by all elements of the 
form $a_\ss-F^\ss_\tau(a_\ss)$ when $\tau\subseteq\ss$
and $a_\ss\in F(\ss)$. This is a right, but not left,
exact functor from sheaves on $X$ to $R$-modules,
and the left derived functors (or higher colimits)
evaluated at $F$ give the \emph{homology
$H_*(X;F)$ of $X$ with coefficients in the 
sheaf $F$}. (The \emph{co}homology is given by the 
right derived functors of the limit functor, but we
will have no need for cohomology in this paper.)

There is a chain complex that computes this
homology; it is essentially the simplicial
chain complex of $X$ suitably modified to
encompass the extra structure contained in the
sheaf.
Fix a partial ordering $\leq$ of the vertices of $X$ that has the property
that for any face $\ss$, the restriction of $\leq$ to $\ss$
is a total ordering. In particular, if
$\ss$ is a $i$-face then $\ss$ can be written
$\ss=x_0<x_1<\cdots< x_i$, or just $\ss=x_0x_1\ldots x_i$.
An $i$-face in this form is said to be \emph{oriented\/}.

The homology $H_*(X;F)$ is computed by the chain complex
$\Sym{*}(X;F)$ whose module of $i$-chains is:
$$
S_{\kern-1pt i}(X;F)=\bigoplus_\ss F(\ss),
$$
the direct sum over the oriented $i$-simplices $\ss$.
The differential $d:S_{\kern-1pt i}\ra S_{\kern-1pt{i-1}}$
is given by:
$$
\lambda\cdot\ss
\mapsto
\sum_{i=0}^{i}(-1)^i 
F_{d_i\ss}^{\ss}(\lambda)\cdot d_i\ss,
$$
where for $\ss=x_0x_1\ldots x_i$, the $i$-th face $d_i\ss$
is obtained by omitting $x_i$ from $\ss$, and the element
$\lambda\cdot\ss\in S_{\kern-1pt i}$ has $\lambda\in F(\ss)$
in the $\ss$-coordinate and $0$'s elsewhere. If $X=\varnothing$
then $S_*$ is the zero complex.

The result is a chain complex with homology $H_*(X;F)$, i.e.
$H S_*(X;F)=H_*(X;F)$.

We conclude this brief summary with the properties
of sheaf homology that we will directly use.
Firstly, a short exact sequence of sheaves
$0\ra H\ra F\ra G\ra 0$ on the simplicial
complex $X$
induces a long exact sequence in homology:
\begin{equation}
\label{equation:sheaves:LES}
    \cdots\ra
    H_{i+1}(X;G)\ra
    H_i(X;H)\ra H_i(X;F) \ra H_i(X;G)
    \ra H_{i-1}(X;H)
    \ra\cdots
\end{equation}
If $A$ is a fixed $R$-module then the
\emph{constant sheaf\/} $\Delta A$ on $X$
has value $A$ at every face $\ss$ and structure
maps that are the identity $A\ra A$. If $R=\Z$,
the situation in which all of our sheaves will
arise, then the sheaf homology $H_*(X;\Delta A)$
is just the ordinary simplicial homology
$H_*(X;A)$ of $X$ with coefficients in (the 
Abelian group) $A$. When $A=\Z$ we will just 
write $H_*X$ for the ordinary integral
simplicial homology. 

Finally, 
the complex 
$S_*(X;\bigoplus_\lambda F_\lambda)\cong\bigoplus_\lambda S_*(X;F_\lambda)$,
so that  
$H_*(X;\bigoplus_\lambda F_\lambda)
\cong \bigoplus_\lambda H_*(X;F_\lambda)$.

%%%%%%%%%%%%%%%%%%%%%%%%%%%%%%%%%%%%%%%%%%%%%%%%%%%%%

\subsection{Order sheaves}
\label{sheaves:order:sheaves}

In this
paper the sheaves have quite a simple structure.

If $X$ is a simplicial complex
then an \emph{order sheaf\/} $F$ on
$X$ has value at the vertex $x$ given by $F(x)=\Z_x:=\Z_{m(x)}$, 
where $m(x)\geq 2$ is an integer.
If $\ss=\{x_0,x_1,\ldots,x_i\}$ is an $i$-face,
then we require that
$m(x_j)$ and $m(x_k)$ are relatively prime for all $j\not = k$, and
that $F(\ss)=\Z_\ss:=\Z_{m(\ss)}$
where $m(\ss)=m(x_0)m(x_1)\ldots m(x_i)$. 
If $\tau\subseteq\ss$ is a face of $\ss$ then 
$m(\tau)=m(x_{i_0})m(x_{i_1})\ldots m(x_{i_\ell})$
for some $0\leq i_0<i_1<\cdots<i_\ell\leq i$. In particular,
$m(\tau)$ divides $m(\ss)$, and the structure
map $F_\tau^\ss:\Z_{m(\ss)}\ra\Z_{m(\tau)}$ is given by reduction modulo
$m(\tau)$. We leave it to the reader to check that if 
$\tau\subseteq\ss\subseteq\mu$ are faces then 
$F^\mu_\tau=F^\ss_\tau F^\mu_\ss$, and so $F$ is indeed a sheaf.

The orders of group elements naturally lead to an order
sheaf on the clique complex $X_G$: the value $F(g)$ for a face
$g\in X_G$ is $\Z_{|g|}$, where $|g|$ is the order of $g$.
We will just write $\Z_g$ from now on. 
In particular, if $g$ is a vertex of $X_G$ 
--- hence a $p$-element for some prime $p$ --- then $F(g)=\Z_{q}$ 
where $q=p^m$, and if $g=g_0g_1\ldots g_i$ is an $i$-face then
$F(g)=\Z_{|g|}$ where $|g|=q_0q_1\ldots q_i$
with the $q_j$ powers of distinct primes.
If $h$ is a face of $g$ in $X_G$ then $h=g_{i_0}g_{i_1}\ldots g_{i_\ell}$
for some $0\leq i_0<i_1<\cdots<i_\ell\leq i$,
so that 
$|h|=q_{i_0}q_{i_1}\ldots q_{i_\ell}$
divides $|g|$. The structure map
$F_h^g:\Z_{g}\ra\Z_{h}$ of the sheaf is then 
reduction modulo $|h|$. 

If $F$ is an order sheaf on a simplicial complex $X$ and $x\in X_0$
is a vertex then the \emph{vertex sheaf\/} $F_x$ is defined by:
\begin{equation}
  \label{vertex:sheaf:equation8457}
F_x(\ss)=
  \left\{
\begin{array}{ll}
\Z_x,  & x\in\ss,\\
0,  &\text{else}.\\
\end{array}
\right.
\end{equation}
Thus, $F_x$ is non-zero on $\St_x\setminus\Lk_x$,
where it has the same value that $F$ has at the vertex $x$.
The structure maps are the identity or zero maps
as appropriate.

\begin{proposition}
  \label{proposition:q097h4}
  Let $X$ be a simplicial complex with vertices $X_0$
  and $F$ an order sheaf on $X$.
  Then there is an isomorphism of sheaves 
  $F\cong\bigoplus_{x\in X_0} F_x$.
\end{proposition}

\begin{proof}
Let $G=\bigoplus_{x\in X_0} F_x$ and let 
 $\ss=\{x_0,x_1,\ldots,x_i\}$ be an $i$-face of $X$.
Then:
\begin{equation}
\label{eq:sumdecomp}
F(\ss) 
= \Z_\ss 
= \Z_{m(x_0)\cdots m(x_i)}
\cong \bigoplus_{j=0}^i \Z_{x_j} 
= \bigoplus_{j=0}^i F_{x_j}(\ss) 
= G(\ss),
\end{equation}
as the $m(x_0),\ldots, m(x_i)$ are mutually relatively prime.
Define $\kappa_\ss:F(\ss)\ra G(\ss)$ to be the isomorphism
(\ref{eq:sumdecomp}). To see that the $\kappa_\ss$ assemble into a
sheaf isomorphism $\kappa:F\ra G$, it remains to show that
for any pair of faces $\tau\subseteq\ss$, we have
$G^\ss_\tau \kappa_\ss=\kappa_\tau F^\ss_\tau$.
Suppose then that 
$\tau = \{y_{0},\ldots,y_{\ell}\}$ is some face of $\ss$, and let
$m(\tau) =m(y_{0})\cdots m(y_{\ell})$ be the corresponding divisor of $m(\ss)$.
As both $G^\ss_\tau \kappa_\ss$ and $\kappa_\tau F^\ss_\tau$
are homomorphisms, their equality need only be checked on the
generator $1$ of $F(\ss) = \Z_{m(\ss)}$.
On the one hand, we have:
$
\kappa_\tau F^\ss_\tau(1) = \kappa_{\tau}(1) = (1,\ldots,1)\in
\bigoplus_{j=0}^\ell \Z_{y_j}, 
$
where there are $\ell+1$ coordinates in the tuple
corresponding to the $\ell+1$
summands of $G(\tau)$.
On the other hand, we have
$\kappa_\ss(1) = (1,\ldots,1)\in \bigoplus_{j=0}^i \Z_{x_j}$, where we
now have $i+1$ coordinates corresponding to the $i+1$ summands of $G(\ss)$. 
Since this element has $1$ in every coordinate, its image under
$G^\ss_\tau$ also has $1$ in every coordinate.
\qed
\end{proof}

\begin{corollary}
\label{corollary:89374fhytqb}
If $F$ is an order sheaf on a simplicial complex $X$ 
with vertices $X_0$, then:
\begin{equation}
\label{equationq90htb7j}
H_*(X;F)\cong
\bigoplus_{x\in X_0}H_*(\St_x;F_x).
\end{equation}
\end{corollary}

This follows because
$H_*(X;{\textstyle{\bigoplus_{x\in X_0}}}F_x)
\cong\bigoplus_{x\in X_0}H_*(X;F_x)$ and
$S_*(X;F_x)=S_*(\St_x;F_x)$, as the vertex sheaf $F_x$ 
has value $0$
outside of the star $\St_x$.

%%%%%%%%%%%%%%%%%%%%%%%%%%%%%%%%%%%%%%%%%%%%%%%%%%%%%

\subsection{The homology in general}
\label{sheaves:general:result}

The sheaf homology of an arbitrary simplicial complex
equipped with an order sheaf can be computed in terms of the
\emph{ordinary\/} homology of the links of vertices. 

\begin{theorem}
  \label{clique:complex:theoremeuyrh9erw}
Let $X$ be a %finite $d$-dimensional 
simplicial complex with vertices
$X_0$, and $F$ an order 
sheaf on $X$. Then for $i\geq 2$ the homology:
$$
H_i(X;F) \cong \bigoplus_{x\in X_0}
H_{i-1}(\Lk_x,\Z_x).
$$
%where $d_x$ is the dimension of the star $\St_x$.
In degree $1$
we have
$H_1(X;F) \cong \bigoplus \Z_x^{m_x-1}$, the sum over the
vertices of $X$, and
where  $m_x$ is the number of connected components of the link
$\Lk_x$ of $x$. In degree $0$ we have
$H_0(X;F) \cong \bigoplus \Z_x$, the sum over the
isolated vertices.
The homology vanishes in all other degrees.
\end{theorem}

Note that the number of connected components of the empty
set is $1$, and so isolated vertices contribute nothing to the 
homology in degrees $>0$.

\begin{proof}
  By Corollary \ref{corollary:89374fhytqb}
  it suffices to compute the homology $H_*(\St_x,F_x)$ for
  each vertex $x$. Let $\Delta\Z_x$ be the constant sheaf
  on the star $\St_x$ and consider the map
  $\kappa:\Delta\Z_x\ra F_x$, where for any $\ss\in\St_x$,
  the map $\kappa_\ss$ is either the identity $\Z_x\ra\Z_x$
  or $\Z_x\ra 0$ as appropriate. Then
  $(F_x)^\ss_\tau\,\kappa_\ss=\kappa_\tau\,(\Delta\Z)^\ss_\tau$ for
  any $\tau\subseteq\ss$ in $\St_x$: there are three cases to check,
  depending on whether $x\in\tau$ or $x\in\ss,x\not\in\tau$
  or $x\not\in\ss$. Thus $\kappa$ is a map of sheaves on $\St_x$.
  Moreover, each $\kappa_\ss$ is a surjection,
  and so $\kappa$ is a surjective sheaf map.
  
  There is thus a short exact sequence
  of sheaves on $\St_x$:
  \begin{equation}
    \label{SES:equation:q08497h57}
  0\ra K\ra \Delta\Z_x\ra F_x\ra 0,
  \end{equation}
  where $K$, the kernel sheaf, has values:
  \begin{equation*}
K(\ss)=
  \left\{
\begin{array}{ll}
\Z_x,  & x\in\Lk_x,\\
0,  &\text{else},\\
\end{array}
\right.
\end{equation*}
  with structure maps $0\ra 0, 0\ra\Z_x$ or $\Z_x\ra\Z_x$.
  In particular, $H_*(\St_x;K)\cong H_*(\Lk_x,\Z_x)$, where now
  the latter homology has constant coefficients $\Z_x$.

  The result is that the short exact sequence
  (\ref{SES:equation:q08497h57}) induces a long exact sequence in
  homology:
  \begin{equation*}
    \label{LES:equation:qs23drhe9g4}
    \cdots\ra
    H_{i+1}(\St_x;F_x)\ra
    H_i(\Lk_x,\Z_x)\ra H_i(\St_x,\Z_x) \ra H_i(\St_x;F_x)
    \ra H_{i-1}(\Lk_x,\Z_x)
    \ra\cdots
  \end{equation*}
  for $0\leq i\leq d_x$, where $d_x$ is the dimension of the
  star $\St_x$. As $\St_x$ is contractible (and in particular,
  $H_*\St_x$ is a free $\Z$-module)
  and the middle terms
  have constant coefficients $\Z_x$, we have $H_i(\St_x,\Z_x)=0$
  when $i>0$ and $H_0(\St_x,\Z_x)\cong\Z_x$. The long exact
  sequence thus decomposes
  into fragments that give $H_i(\St_x;F_x)\cong H_{i-1}(\Lk_x,\Z_x)$
  when $2\leq i\leq d_x$, as well as an exact sequence:
  \begin{equation}
    \label{ES:equation:9w6rj7}
  0\ra H_1(\St_x;F_x)\ra H_{0}(\Lk_x,\Z_x)\ra\Z_x\ra H_0(\St_x,;F_x)\ra 0.
  \end{equation}
  If $m_x$ is the number of connected components of the link $\Lk_x$
  then $H_{0}(\Lk_x,\Z_x)\cong\Z_x^{m_x}$, as $H_0(\Lk_x)$ is free.

  Suppose now that $x$ is not an isolated vertex. Then as the vertex sheaf
  $F_x$ has value $0$ on any vertex of $\St_x$ that is not $x$, 
  the $0$-chains $\Sym{0}(\St_x;F_x)=\Z_x=F_x(x)$. Let $x_0x_1$ be an edge
  of $\St_x$, with $x$ one of $x_0$ or $x_1$, and let 
  $\lambda\in\Z_x=F_x(x)$. Then the element of $\Sym{1}(\St_x;F_x)$,
  having value $\lambda$ in the component indexed by the edge $x_0x_1$
  and $0$ in all other components, maps via the differential
  $\Sym{1}(\St_x;F_x)\ra \Sym{0}(\St_x;F_x)$ to $\pm\lambda$.
  The differential is thus onto and so $H_0(\St_x;F_x)=0$.
  The exact sequence (\ref{ES:equation:9w6rj7}) is then short:
  $$
  0\ra H_1(\St_x;F_x)\ra \Z_x^{m_x}\ra\Z_x\ra 0,
  $$
  from which we can deduce that $H_1(\St_x;F_x)\cong\Z_x^{m_x-1}$.

  If on the other hand $x$ is an isolated vertex, then the
  complex $\Sym{*}(\St_x;F_x)$ is non-zero only in degree $0$,
  where it is
  $F_x(x)=\Z_x$; the homology $H_i(\St_x;F_x)$
  is thus $\Z_x$ in degree $0$ and vanishes in all other degrees. 
  Note that both $\Z_x^{m_x-1}$ and $H_*(\Lk_x;F_x)$ vanish
  for an isolated vertex.
  
  Assembling (\ref{equationq90htb7j}) then gives the advertised result. 
\qed  
\end{proof}

If $X$ is $d$-dimensional for some $d$, then $H_i(X;F)$ vanishes 
when $i>d$.

%%%%%%%%%%%%%%%%%%%%%%%%%%%%%%%%%%%%%%%%%%%%%%%%%%%%%

\subsection{The homology of the clique complex equipped 
with the order sheaf}
\label{sheaves:clique:complexes}

Theorem \ref{clique:complex:theoremeuyrh9erw} can be
interpreted for a clique complex equipped with the order sheaf.
Recall, for $g\in G$ a $p$-element, that $X_{C_g}^p$ is
the $p$-restricted clique complex of the centralizer $C_g$ of $g$:

\begin{corollary}
  \label{corollary:clique:complex:general}
  Let $G$ be a finite group and $F$ the order sheaf on the
  clique complex $X_G$. Then
\begin{enumerate}
    \item In degrees $2\leq i\leq \dim G$ 
    the homology:
$$
H_i(X_G;F) \cong \bigoplus_{g}
H_{i-1}(X_{C_g}^p,\Z_g),
$$
the sum over the $g\in G$ that are $p$-elements
for some prime $p$.
    \item In degree $1$ the homology 
    $H_1(X_G;F)\cong\bigoplus \Z_g^{m_g-1}$,
    the sum over the $g\in G$ that are $p$-elements
for some prime $p$ and where $m_g$ is the number of 
connected components of $X_{C_g}^p$.
    \item In degree $0$ the homology $H_0(X_G;F)\cong\bigoplus\Z_g$,
    the sum over the $g\in G$ that are $p$-elements
for some prime $p$ and that have centralizer $C_g$ a $p$-group.
\end{enumerate}
The homology vanishes in all other 
degrees.
\end{corollary}

The result is a straight translation of 
Theorem \ref{clique:complex:theoremeuyrh9erw},
combined with Lemma \ref{lemma:links}.
The corollary can be implemented when the centralizers of elements 
have clique complexes that are well understood --- see 
Section \ref{section:examples} for some examples of this. 
In the meantime, it is also possible to describe the sheaf homology
in low dimensions:

\begin{corollary}
    \label{corollary:clique:complex:dimension0}
Let $G$ be a group of dimension $0$, i.e. an
EPPO-group, and let
$F$ be the order sheaf on $X_G$. Then:
$$
H_0(X_G;F) \cong \bigoplus_{g}
\Z_g,
$$
the sum over the elements $g\in G$, and $H_i(X_G;F)=0$
when $i\not=0$.
\end{corollary}

The result follows immediately, as does:

\begin{corollary}
    \label{corollary:clique:complex:dimension1}
Let $G$ be a group of dimension $1$ and let
$F$ the order sheaf on $X_G$. Then:
$$
H_1(X_G;F) \cong \bigoplus_{g}
\Z_g^{n_g-1},
$$
the sum over the $g\in G$ that are $p$-elements
for some prime $p$ and whose 
centralizer is not a $p$-group.
%and for which the set 
%of $p^\prime$-elements
%$h$ such that $p^\prime$ is a prime $\not=p$ and $hg=gh$ is 
%non-empty. 
In this case $n_g$ is the number
of $p^\prime$-elements
$h$ such that $p^\prime$ is a prime $\not=p$ and $hg=gh$.
%and where $n_g$ is the number of $p^\prime$-elements
%$h$ such that $p^\prime$ is a prime $\not=p$ and $hg=gh$.
In degree $0$ we have $H_0(X_G;F)\cong\bigoplus\Z_g$, the sum
over the $g\in G$ that are $p$-elements
for some prime $p$ and have centralizer a $p$-group.
The homology $H_i(X_G;F)=0$
when $i\not=0,1$.
\end{corollary}

%The $p$-restricted clique complex of the centralizer %$X_{C_g}^p$
%of $g$ is empty if and only if the centralizer %$C_g$ 
%is a $p$-group. 
%In this case $m_g-1=0$ but $n_g-1=-1$; this is the reason for
%the extra restriction on the type of $g$ that contribute
%to $H_1(X_G;F)$ in Corollary \ref{corollary:clique:complex:dimension1}.

%In general the number $m_g$ of connected components of $
%X_{H}^p\,(H=C_g)$
%is at most $n_g$, with equality exactly when $X_{H}^p$
%is $1$-dimensional. 
If $G$ has dimension $>1$ and $g$ is a vertex
of an $i$-face with $i>1$, then the other $i$ vertices of this
face lie in the same connected component of $\Lk_g=X_{H}^p$
for $H=C_g$.
These vertices thus contribute $i$ to the
$n_g$ of Corollary \ref{corollary:clique:complex:dimension1}, 
but only $1$ to the $m_g$ of 
Corollary \ref{corollary:clique:complex:general}.
The homology $H_1(X_G;F)$ thus has $\Z_g$-rank strictly
less than what is advertised in Corollary 
\ref{corollary:clique:complex:dimension1} when $\dim G>1$
and $g$ lies in an $i$-face for $i>1$. For example in
$G=\Sym{10}$, the smallest symmetric group
of dimension $2$, the element $g=(1,2)$ has $h_1=(3,4,5)$
and $h_2=(6,7,8,9,10)$, contributing $2$ to $n_g$, but
the $2$-cell $(1,2)(3,4,5)(6,7,8,9,10)$ in $X_G$ means they 
contribute $1$ to $m_g$.

%%%%%%%%%%%%%%%%%%%%%%%%%%%%%%%%%%%%%%%%%%%%%%%%%%%%%
%%%%%%%%%%%%%%%%%%%%%%%%%%%%%%%%%%%%%%%%%%%%%%%%%%%%%

\section{$H_*(X_G;F)$ for certain families of groups}
\label{section:examples}

In this section we apply Corollaries 
\ref{corollary:clique:complex:general}-\ref{corollary:clique:complex:dimension1}
to compute the sheaf homology $H_*(X_G;F)$ for certain $G$.
There are three possible routes to calculation: the first is to use
brute force for the groups $G$ with $\dim G\leq 1$. As this only
requires knowledge of the conjugacy classes of $G$ we could in 
principle compute the sheaf homology of $X_G$
for many of the sporadic simple groups; we content ourselves, by
way of example, with the smallest of the Mathieu groups $M_{11}$.
The second route (which leads to the third) 
is to investigate those $G$ for which the $p$-restricted
clique complex $X^p_{C_g}$ of centralizers
is well understood. 
For us, this will mean the nilpotent groups,
for which we avail ourselves of the theory of shellings. 
As the key ingredient in the calculation of $H_*(X_G;F)$ is the
ordinary homology of the centralizers of elements, the groups
for which these centralizers are nilpotent then naturally follow.

%%%%%%%%%%%%%%%%%%%%%%%%%%%%%%%%%%%%%%%%%%%%%%%%%%%%%

\subsection{Low dimensional examples}

When $G$ has dimension $\leq 1$ it is possible to 
explicitly compute
the homology $H_*(X_G;F)$ by brute force
using Corollaries 
\ref{corollary:clique:complex:dimension0}-\ref{corollary:clique:complex:dimension1}
and some knowledge of the conjugacy classes in $G$. 
As an illustration we continue the analysis from
Section \ref{section:clique:complex:sporadic},
except for the 
Mathieu group $M_{11}$. The assertions below follow
from a combination of the $\ams{ATLAS}$ \cite{MR0827219}
and {\sc Gap} \cite{gap}.

\begin{proposition}
\label{sheaf:homology:Mathieu:11}
Let $G$ be the Mathieu group $M_{11}$ of order
$m=2^4\cdot 3^2\cdot 5\cdot 11$
and $F$ the order sheaf on $X_G$. Then:
$$
H_0(X_G;F)\cong \Z_4^{n_4}\oplus\Z_5^{n_5}\oplus\Z_8^{n_8}
\oplus\Z_{11}^{n_{11}}
\text{ and }
H_1(X_G;F)\cong \Z_2^{n_2}\oplus\Z_3^{n_3},
$$
where $n_2=7m/(2^4\cdot 3),n_3=m/3^2,
n_4=m/8,
n_5=m/5,n_8=m/4$ and $n_{11}=2m/11$.
\end{proposition}

\begin{proof}
In $\ams{ATLAS}$ notation, the group $M_{11}$ has two conjugacy classes
of $p$-elements, $2A$ and $3A$,
where the centralizers are not $p$-groups, and six
others, $4A,5A,8A,8B,11A$ and $11B$, where they are.
These latter six classes have sizes $m/8,m/5,m/8,m/8,m/11$ and $m/11$,
from which the degree $0$ homology follows using
Corollary \ref{corollary:clique:complex:dimension1}.
In the natural representation $M_{11}$ has generators:
$$
a=(2,10)(4,11)(5,7)(8,9) \text{ and } b=(1,4,3,8)(2,5,6,9):
$$
and conjugacy class representatives $a\in 2A$
and $ab^2ab^2\in 3A$. The centralizer of $a$
contains $8$ elements of order $3$, and no other $p^\prime$
elements, for $p^\prime$ a prime $\not=2$; the centralizer of $ab^2ab^2$
contains $3$ elements of order $2$. This, combined with
the sizes of $2A$ and $3A$ --- $m/48$ and $m/18$ --- gives the homology
in degree $1$.
\qed
\end{proof}

%%%%%%%%%%%%%%%%%%%%%%%%%%%%%%%%%%%%%%%%%%%%%%%%%%%%%

\subsection{Shellings}

The theory of shellings is a simple yet powerful
tool for calculating the homology of simplicial 
complexes. In this section we see that if the complex
$X$ can be shelled, then not only does this yield the 
ordinary homology $H_*X$, but also the links $\Lk_x$
of vertices can be shelled, and hence \emph{their\/} ordinary homologies
can be computed. Theorem \ref{clique:complex:theoremeuyrh9erw}
then gives the \emph{sheaf\/} homology of $X$ with coefficients
in an order sheaf. This idea is particularly effective for the 
clique complexes $X_G$ of the nilpotent groups $G$.

With this in mind, a
\emph{shelling\/} of a finite dimensional (not necessarily pure)
simplicial complex $X$ is a 
total ordering $\sigma_1,\sigma_2,\ldots,\sigma_m$ of the
facets such that for all $t>1$:
\begin{equation}
\label{eq5}
    \overline{\sigma}_t
\cap
{\textstyle \bigcup}_{s<t}\,\overline{\sigma}_s,
%{\textstyle \bigcup}_{i=1}^{t-1}\,\overline{\sigma}_i,
\end{equation}
is a \emph{pure\/} $(\dim \sigma_t-1)$-dimensional subcomplex of 
the simplex $\overline{\sigma}_t$. In this case define:
\begin{equation}
\label{eq6}
\mathscr{R}(\sigma_t)
=\{x\in{\sigma}_t:
{\sigma}_t\,\setminus\, x\in 
{\textstyle \bigcup}_{s<t}\,\overline{\sigma}_s\},
%{\textstyle\bigcup}_{i=1}^{t-1}\overline{\sigma}_i\},
\end{equation}
and call $\ss$ a \emph{homology facet} when 
$\mathscr{R}(\sigma)=\sigma$. 
The ordinary \emph{reduced\/} homology
$\widetilde{H}_i X$ is then isomorphic to the free $\Z$-module with
basis the $i$-dimensional homology facets
--- see \cite{MR1333388}. 
The unreduced homology $H_i X$ is thus $\widetilde{H}_i X$
when $i>0$, and the free $\Z$-module on the set of connected
components of $X$ when $i=0$.
We use the
theory of shellings for non-pure
complexes in general, although in the 
application below the complexes are pure
and so it makes no difference.

Shellability is a property that is inherited by links.
For, if $\ss$ is a face of $X$ and $\tau$ is a facet of
the link $\Lk_\ss$ then $\tau\cup\ss$ is a facet of $X$.
If $\sigma_1,\sigma_2,\ldots,\sigma_m$ is a shelling of $X$
then $\tau\cup\ss$ is one of the $\ss_t$. Thus the facets of
$\Lk_\ss$ can be totally ordered $\tau_1,\tau_2,\ldots,\tau_n$
so that if $\tau_j\cup\ss=\ss_{i_j}$ then
$i_1<i_2<\cdots<i_n$. By \cite{MR1401765}*{Proposition 10.14}
this total ordering is then a shelling of $\Lk_\ss$.

The clique complexes of nilpotent groups, and more
generally the complexes $X(a_1,a_2,\ldots,a_n)$ defined in
Section \ref{section:clique:complex:nilpotent} can be shelled:

\begin{proposition}
  \label{clique:nilpotent:homology1}
The reduced homology:
$$
\widetilde{H}_{n-1} \,X(a_1,a_2,\ldots,a_n)\cong \Z^{(a_1-1)\ldots(a_n-1)},
$$
and $\widetilde{H}_i\,X(a_1,a_2,\ldots,a_n)$ vanishes for all 
$i\not= n-1$.
\end{proposition}

 \begin{proof}
$X(a_1)$ consists of $a_1$ isolated vertices and so the result
is immediate. For $n>1$, order
the facets of $X(a_1,a_2,\ldots,a_n)$ lexicographically,
so that 
$(x_1,x_2,\ldots,x_n)<(y_1,y_2,$ $\ldots,y_n)$ iff 
for some $\ell$ we have 
$x_1=y_1,\ldots,x_\ell=y_\ell$ and $x_{\ell+1}<y_{\ell+1}$.
The ordering in the $j$-th component is the usual one in the 
integers $0\leq x_j\leq a_j$.
If $\sigma_t=(x_1,x_2,\ldots,x_n)$ then 
$\overline{\tau}$ is a subsimplex of $\overline{\sigma}_t$
exactly when $\tau=(y_1,y_2,\ldots,y_n)$ and there are
$1\leq i_1<i_2<\cdots<i_k\leq n$ such that 
$y_{i_1}=y_{i_2}=\cdots=y_{i_k}=0$ and 
in all other coordinates
the $y_j$ are
equal to the $x_j$.
Then:
\begin{equation}
    \overline{\tau}\in
    {\textstyle \bigcup}_{s<t}\,\overline{\sigma}_s
    %\bigcup_{i<t}\,\overline{\sigma}_i
    \text{ if and only if }
    x_{i_j}>1\text{ for at least one }x_{i_j}.
\end{equation}
For, if an $x_{i_j}>1$ then replacing it by $x_{i_j}-1$ gives
a new facet $\sigma_s$ with $\overline{\tau}$ a subsimplex of 
$\overline{\sigma}_s$ and $\sigma_s<\sigma_t$. Conversely, 
if all the $x_{i_j}=1$ then $\overline{\tau}$ is a 
subsimplex
only of the facets with the property that
all the $x_{i_j}\geq 1$, and these are
$\geq\sigma_t$ in the total ordering.

The result is that (\ref{eq5}) is the union of the maximal 
subsimplices of $\overline{\sigma}_t$ that are obtained 
by replacing an $x_i>1$ with $x_i=0$ in
$(x_1,x_2,\ldots,x_n)$.
Hence $X(a_1,a_2,\ldots,a_n)$ is shellable,
and 
$\mathscr{R}(\sigma_t)=\sigma_t$ if and only if
$x_i>1$ for all $1\leq i\leq n$.
As $X(a_1,a_2,\ldots,a_n)$ is connected, pure
$(n-1)$-dimensional, it has
homology in top degree $i=n-1$ that is free of rank
$(a_1-1)(a_2-1)\cdots(a_n-1)$. The
homology vanishes in all other degrees.
   \qed
 \end{proof}

\begin{corollary}
\label{corollary:ordinary:nilpotent}
Let $G$ be a nilpotent group of order $q_1q_2\ldots q_n$,
with the $q_i$ powers of distinct primes. 
Then:
$$
\widetilde{H}_{n-1}\,X_G\cong \Z^{(q_1-2)\ldots(q_n-2)}.
$$
 \end{corollary}

 The result follows immediately from the fact that
 $X_G\cong X(q_1-1,q_2-1,\ldots,q_n-1)$.

 Recall that for $k$ a field, 
 the Stanley-Reisner ring $k[X]$ of a simplicial complex $X$
 is the quotient of the polynomial ring
 $k[X_0]$ on the vertices by the ideal generated by 
 the monomials $x_{i_1}x_{i_2}\ldots x_{i_r}$ where
 $\{x_{i_1},x_{i_2},\ldots, x_{i_r}\}$ is \emph{not\/}
 a face of $X$ --- see \cite{MR1453579}*{Chapter II}.
 A simplicial complex is
 \emph{Cohen-Macaulay\/} (over $\Z$) when its Stanley-Reisner
 ring is Cohen-Macaulay --- see \cite{MR1251956}. 
 Equivalently, for any face $\ss$,
 the link $\Lk_\ss$ has ordinary homology that vanishes
 in all degrees except degree $0$ and degree $\dim\Lk_\ss$.
 As pure shellable complexes are Cohen-Macaulay \cite{MR0570784}
 we have:

 \begin{corollary}
The clique complex of a nilpotent group is
Cohen-Macaulay.
 \end{corollary}

We also remark that the homology in Proposition
\ref{clique:nilpotent:homology1} can be computed without 
recourse to shellings: an 
application of the Meyer-Vietoris long exact sequence 
expresses the reduced homology of a join $X\ast Y$ in terms
of the reduced homologies of $X,Y$ and $X\times Y$. An
inductive argument then gives the result.
We prefer the shelling though, as it is the computational
route outlined in the first paragraph of this section 
that we wish to emphasise.

We now proceed to the homology when the complex is endowed
with an order sheaf. Again, we can start in the more
general setting of the complexes $X(a_1,a_2,\ldots,a_n)$:

  \begin{theorem}
   Let $X=X(a_1,a_2,\ldots,a_n)$ 
   and $F$ an order sheaf
   on $X$. Then the homology $H_i(X;F)$ vanishes outside
   degree $i=n-1$ where:
       \begin{equation}
       \label{equation:sheaf:homology:nilpotent}
        H_{n-1}(X;F)\cong
        \bigoplus_{x\in X_0}\Z_x^{(a_1-1)\ldots\widehat{(a_j-1)}\ldots (a_n-1)}
       \end{equation}
       the sum over the vertices
       $x=(0,\ldots,0,x_j,0,\ldots,0)$ and where
       $\widehat{a}=1$.
       %If $X=X(a_1)$ then 
       %$H_0(X;F)\cong\bigoplus_{1\leq x\leq a_1}\Z_x$.
 \end{theorem}

 \begin{proof}
 If $x=(0,\ldots,0,x_j,0,\ldots,0)$ is a vertex with 
 $x_j$ fixed then the link
 $\Lk_x$ consists of the $(x_1,\ldots,x_{j-1},0,x_{j+1},\ldots,x_n)$
 with the $x_k$ arbitrary when $k\not= j$. The shelling 
 in Proposition \ref{clique:nilpotent:homology1} induces 
 a shelling of the facets of $\Lk_x$ which is also 
 the lexicographic order of these tuples. 
 Alternatively, $\Lk_x$ is clearly the complex
 $X(a_1,\ldots,\widehat{a}_j,\ldots,a_n)$. Either way,
 the result can be assembled using 
 Theorem \ref{clique:complex:theoremeuyrh9erw} and
 Proposition \ref{clique:nilpotent:homology1}.

 If $n>2$ then the homology in degree $i$, for
 $2\leq i\leq\dim X=n-1$, is the sum over the vertices
 $x$ of the homologies $H_{i-1}(\Lk_x,\Z_x)$, where
 each such is the summand given in 
 (\ref{equation:sheaf:homology:nilpotent}) when the degree 
 $i=n-1$, and $0$ elsewhere. 
 Also, the links $\Lk_x$ are connected and there are no isolated
 vertices, so the homology in degrees $0$ and $1$ vanishes.
 When $n=2$ the links $\Lk_x$ have $a_1$ or $a_2$ connected components
 and the degree $1$ homology is (\ref{equation:sheaf:homology:nilpotent})
 and degree $0$ vanishes. Finally, when $n=1$ the homology
 in degree $0$ is $\bigoplus_{X_0}\Z_x$ --- and this is also
 (\ref{equation:sheaf:homology:nilpotent}) --- and $0$ in other degrees.
   \qed
 \end{proof}

\begin{corollary}
\label{corollary:sheaf:nilpotent}
Let $G$ be a nilpotent group of order $q_1q_2\ldots q_n$,
with the $q_i$ powers of distinct primes,
and let $F$ be the order sheaf on $X_G$.
Then:
\begin{equation*}
       %\label{equation:sheaf:homology:nilpotent}
        H_{n-1}(X_G;F)\cong
        \bigoplus_{j=1}^n
        \bigoplus_{g}\Z_g^{(q_1-2)\ldots\widehat{(q_j-2)}\ldots (q_n-2)}
       \end{equation*}
       with the second sum over the $p_j$-elements $g$
       and where $\widehat{a}=1$.
       %If $n=1$ then $H_0(X_G;F)\cong
       %\bigoplus\Z_g$, the sum over all non-trivial elements of $G$.
       The homology vanishes in all other degrees. 
\end{corollary}

%%%%%%%%%%%%%%%%%%%%%%%%%%%%%%%%%%%%%%%%%%%%%%%%%%%%%

\subsection{$CN$-groups}

Corollary \ref{corollary:clique:complex:general} 
shows that to compute the sheaf 
homology of the clique complex $X_G$ we need 
to know the ordinary homology of the ($p$-restricted)
clique complex of the centralizers of elements. 
By Corollary \ref{corollary:ordinary:nilpotent} 
we know the
ordinary homology of $X_G$ (and $X_G^p$) when
$G$ is a nilpotent group. This naturally leads us
to consider those groups for which the centralizers
of all non-identity elements are nilpotent. 

Such groups are called $CN$-groups.
It turns out that 
the class of $CN$-groups is quite special; nevertheless, they played
an important early role in the classification of finite
simple groups. Of particular note is the
Feit-Hall-Thompson theorem of 1960: any non-soluble $CN$-group
has even order; see \cite{MR0114856}. 
This presaged
Feit and Thompson's
more famous 1963 result on non-abelian finite simple groups.

The simple $CN$-groups were classified by Suzuki in
\cites{MR0131459,MR0136646}.
They are:
\begin{enumerate}[(i).]
    \item the projective special linear groups 
    $\mathrm{PSL}_2(q)$ with $q=2^m$;
    \item the $\mathrm{PSL}_2(p)$ with $p$ a 
    Fermat prime or a Mersenne prime;
    \item the Suzuki groups $\mathrm{Suz}(2^{2m+1})$;
    \item the groups $\mathrm{PSL}_2(9)$ and $\mathrm{PSL}_3(4)$,
\end{enumerate}
and, of course, the $\Z_p$ for $p$ a prime.
Recall that a Fermat prime is a prime of the form
$2^m+1$ (in which case $m$ itself must also be a 
power of $2$) and a Mersenne prime has the form $2^m-1$,
in which case $m$ is prime. 
There are only finitely many known examples 
of either kind of prime; it is conjectured that 
there are in fact only finitely many Fermat
primes %(namely, $2^1+1,2^2+1,2^4+1,2^8+1$ and $2^{16}+1$)
but there are infinitely many Mersenne
primes.

Combining Corollaries 
\ref{corollary:clique:complex:general} 
and \ref{corollary:ordinary:nilpotent},
we get:

\begin{theorem}
  \label{theorem:general:result:CNgroups}
  Let $G$ be a $CN$-group and $F$ the order sheaf
  on $X_G$. Then
      in degrees $0\leq i\leq \dim G$,
      the homology:
      \begin{equation}
      \label{homology:CN:groups:equation100}
      H_i(X_G;F) \cong \bigoplus_{g}
      \Z_g^{(q_0-2)\ldots\widehat{(q_j-2)}\ldots (q_{i}-2)}
      \end{equation}
      the sum over the $p$-elements $g$
    for some prime $p$ with $|C_g|=q_0q_1\ldots q_{i}$
    a product of distinct prime powers, $q_j$ a power of $p$
    and $\widehat{a}=1$.
  The homology vanishes in all other 
degrees.
\end{theorem}

If $G$ is a $CN$-group has non-trivial centre then $G$ is nilpotent.
In this case Theorem \ref{theorem:general:result:CNgroups} reduces
to Corollary \ref{corollary:sheaf:nilpotent}.

\begin{proof}
The theorem follows by Corollary \ref{corollary:clique:complex:general}
and Proposition \ref{clique:nilpotent:homology1}. In degrees
$2\leq i\leq \dim G$ a $p$-element with
centralizer order $q_0q_1\ldots q_{i}$ has $p$-reduced clique
complex:
$$
X_H^p\cong X(q_0-1,\ldots,\widehat{q_j-1},\ldots,q_{i}-1),
$$
when it is $q_j$ that is a power of $p$. Each such element
then contributes a summand as stated in the theorem.
The calculations, although not the answers, are slightly different
in degrees $0$ and $1$. 
If $g$ is a $p$-element with centralizer as above then for 
$i=0$ or $i\geq 2$ the complex $X_{C_g}^p$ is connected
and so $m_g=1$. The only contribution to the homology in degree $1$
is thus when $C_g$ has order $q_0q_1$, in which case $m_g$
is either $q_0-2$ or $q_1-2$, and again the summand is as given
in the theorem. The homology in degree $0$ is the sum of the
$\Z_g$ for those $p$-elements $g$ whose centralizer has order $q_0=p^m$;
each such is precisely $\Z_g^{\widehat{q_0-2}}$.
\qed
\end{proof}

Theorem \ref{theorem:general:result:CNgroups} can be interpreted
for any group, and indeed
holds in greater
generality than just the $CN$-groups. In degree $0$ the theorem
holds for all groups: the expression (\ref{homology:CN:groups:equation100})
in Theorem \ref{theorem:general:result:CNgroups} reduces in degree
$0$ to that of Corollary \ref{corollary:clique:complex:general}(iii).
By drawing the schematic $\X_G$, or otherwise, the sheaf homology
$H_*(X_G;F)$ can be computed for $G=\mathrm{PSL}_2(11)$; this 
is not a $CN$-group, but nevertheless has homology 
given by Theorem \ref{theorem:general:result:CNgroups}.

On the other hand, (\ref{homology:CN:groups:equation100}) is not necessarily
the sheaf homology in degree $i>0$ for a non $CN$-group. 
For example, the Mathieu group
$M_{11}$ has exactly two conjugacy classes of elements having
centralizers of orders $q_0q_1$, namely $2A$ --- a conjugacy class of size 
$3\cdot 5\cdot 11$ with centralizers of order $2^4\cdot 3$
--- and $3A$, of size $2^3\cdot 5\cdot 11$ and with order $2\cdot 3^2$; 
see \cite{MR0827219}. The expression 
(\ref{homology:CN:groups:equation100}) gives $\Z_2^{3\cdot 5\cdot 11}$
in degree $1$, whereas the sheaf homology in degree $1$
is $\Z_2^{n_2}\oplus\Z_3^{n_3}$ where
$n_2=7m/(2^4\cdot 3),n_3=m/3^2$ and $m=2^4\cdot 3^2\cdot5\cdot 11$, as
given in 
Proposition \ref{sheaf:homology:Mathieu:11}.

\vspace{1em}

We can explicitly compute the sheaf homology for the simple 
$CN$-groups listed above. Recall that $\omega(n)$ is the 
number of distinct prime divisors of the positive integer $n$,
and let $\Omega(n)$ denote the \emph{set\/}
of prime divisors of $n$.
We discussed the conjugacy classes in the group $G=\mathrm{PSL}_2(q)$ with $q=2^m$, listing them in 
Table \ref{table:PSL:q:2} above.
This shows that:
$$
\dim\mathrm{PSL}_2(q)+1=\max\{1,\omega(q-1),\omega(q+1)\},
$$
with the right hand side equal to $1$ exactly when both
$2^m-1$ and $2^m+1$ are prime powers. 
But then recall that Catalan's conjecture \cite{MR2076124} shows this is impossible unless exactly one is a prime
(and this only happens when $m=3$) or when
$m=2$, in which case both
$2^m-1$ and $2^m+1$ are primes. But if $m\geq 4$
and $2^m+1$ is prime, then necessarily 
$m=2^t$ and $2^{2^t}-1=(2^{2^{t-1}}-1)(2^{2^{t-1}}+1)$ is not prime.
%If the five Fermat primes 
%currently known (namely: $2^1+1,2^2+1,2^4+1,2^8+1$ and
%$2^{16}+1$) are indeed \emph{all\/} the Fermat primes, 
%then $\mathrm{PSL}_2(4)$ is the only $0$-dimensional
%group in this family, and for $m>2$ we have
The upshot is that for $m\geq 4$ we have 
$\dim\mathrm{PSL}_2(2^m)+1=\max\{\omega(2^m-1),\omega(2^m+1)\}$.

\begin{corollary}
\label{corollary:PSL_2:characteristic2:groups}
Let $G=\mathrm{PSL}_2(q)$, where $q=2^m\,(m\geq 2)$,
let $F$ be the order sheaf on $X_G$
and let $c=q\pm 1$.
Then
$H_*(X_G;F)$ is non-vanishing in degrees
$0$ and $\omega(c)-1$.
\begin{enumerate}[(i).]
\item Write out the prime factorisation
$c=\prod_{\Omega(c)} s^{m(s)}$.
Let $r=\frac{q-2}{2}$ when $c=q-1$ and $r=\frac{q}{2}$
when $c=q+1$. Let $\Lambda(c)$ be the set of $\ell$ in the
range $0\leq\ell\leq r$ for which $\gcd(c,\ell)=c/t^{k}$
for some
$t=t(\ell)\in\Omega(c)$
and $k=k(\ell)$ with $0\leq k\leq m(t)$.
Then for 
$i=\omega(c)-1$, the homology $H_i(X_G;F)$ has summand
$\bigoplus_{\Lambda(c)}\Z^{a(\ell)}_{b(\ell)}$,
where $b(\ell)=t^{k}$, and:
$$
a(\ell)=q(q\mp 1)\prod (s^{m(s)}-2),
%{\textstyle a(t)=q(q\mp 1)\prod (s^{m(s)}-2)},
$$
the product over the $s\in\Omega(c)$ such that $s\not=t$.
\item The homology $H_0(X_G;F)$ has summand $\Z_2^{(q-1)(q+1)}$.
\end{enumerate}
\end{corollary}

The three cases in the corollary are not necessarily exclusive.
Again by Catalan's conjecture,
if $q-1$ is a prime power
then in fact it is 
a (Mersenne) prime $s$. Then 
$\omega(q-1)-1=0$ and $H_0(X_G;F)$ is a direct sum with
both $\Z_2^{(q-1)(q+1)}$
and $\Z_s^{q(q+1)}$ components. Similarly if $q+1$ is
a prime power then it is either $9$ or
a (Fermat) prime $s$, and $H_0$ has additional $\Z_s^{q(q-1)}$
components. 
As mentioned above, the case $m=2$ is the only
one where \emph{both\/} $q-1$ and $q+1$ are prime. In this case
all the homology is concentrated in degree $0$.

\begin{proof}
The corollary follows from the data in Table \ref{table:PSL:q:2}; we briefly recall the details. 
%(This theorem is for $\mathrm{SL}_2$ in characteristic $2$,
%which is the same group as $\mathrm{PSL}_2$.)
The group
$G$ has $q$ non-identity conjugacy classes: there is a single class of 
an element $D$ of order $2$, as well as $\frac{q-2}{2}$ classes 
from powers
of an element $A$ of order $q-1$, and $\frac{q}{2}$ classes from powers
of an element $B$ of order $q+1$. 

We now apply Theorem \ref{theorem:general:result:CNgroups}.
That $H_*(X_G;F)$ has only three non-vanishing degrees follows
as there are only three possible sizes of non-identity centralizers. 
Each of the $(q-1)(q+1)$
elements of order $2$ contributes a $\Z_2$ to $H_0$.
The element $A^\ell$ has prime power order for the 
$\ell$ in $\Lambda(c)$ with $c=q-1$,
and for such $\ell$ 
the element has order $t^k$.
Each of the $q(q+1)$ elements in this conjugacy
class thus contributes $\prod_{s\not= t} (s^{m(s)}-2)$
%\begin{equation}
%\label{PSL2:char2:equation100}
%\prod_{s\not= t} (s^{m(s)}-2),
%%(u_1-2)\ldots\widehat{(u_r-2)}\ldots(u_{w_1}-2),
%\end{equation}
copies of $\Z_b$, by Theorem 
\ref{theorem:general:result:CNgroups}(i), whence
the homology in degree $i=\omega(q-1)-1$.
(If $\omega(q-1)=1$ then $\prod_{s\not= t} (s^{m(s)}-2)$
is just $1$, and there are $q(q+1)$ copies of $\Z_b$
in degree $0$.) The exact same analysis applied to the
elements $B^\ell$ gives the advertised homology in 
degree $\omega(q+1)-1$.
\qed
\end{proof}

The special linear group
$G=\mathrm{PSL}_2(q)$, when $q$ is a power of an odd prime,
has order $\frac{1}{2}q(q-1)(q+1)$ and
conjugacy class data that we previously gave in Table \ref{table:PSL:q:odd}. 
The proof of the following is analogous to that for Corollary
\ref{corollary:PSL_2:characteristic2:groups}.

\begin{corollary}
\label{corollary:PSL_2:characteristicp:groups}
Let $G=\mathrm{PSL}_2(p)$, for $p$ a Mersenne or
Fermat prime, 
and let $F$ be the order sheaf on 
$X_G$. Let $c=\frac{1}{2}(p\pm 1)$ with a ``$+$'' if
$p$ is Fermat and a ``$-$'' if $p$ is Mersenne.
%if $p$ is a Mersenne prime
%or $c=\frac{1}{2}(p+1)$ if $p$ is a Fermat prime.
Then 
$H_*(X_G;F)$ is non-vanishing in degrees
$0$ and $\omega(2c)-1$.
\begin{enumerate}[(i).]
\item Write 
the prime factorisation
$c=\prod_{\Omega(c)} s^{m(s)}$ and
let $r=\frac{p-3}{2}$ when $c=\frac{1}{2}(p-1)$, or $r=\frac{p-1}{4}$
when $c=\frac{1}{2}(p+1)$. 
Let $\Lambda(c)$ be the set of $\ell$ in the
range $0\leq\ell\leq r$ for which $\gcd(c,\ell)=c/t^{k}$
for some
$t=t(\ell)\in\Omega(c)$
and $k=k(\ell)$ with $0\leq k\leq m(t)$.
Then for 
$i=\omega(2c)-1$, the homology $H_i(X_G;F)$ has summand
$\bigoplus_{\Lambda(c)}\Z^{a(\ell)}_{b(\ell)}$,
where $b(\ell)=t^{k}$, and:
$$
a(\ell)=p(p\mp 1)\prod (s^{m(s)}-2),
$$
the product over the $s\in\Omega(c)$ such that $s\not=t$.
\item $H_0(X_G;F)$ has summands:
$$\Z_2^{\frac{p^2\pm p}{2}},
\,\,\,
\Z_p^{p^2-1}
\text{ and }
\bigoplus_{\ell=1}^{u}
\Z_{n(\ell)}^{p^2\pm p},
$$
where $n(\ell)=[\frac{p\mp 1}{2},\ell]$,
and $u=\frac{p-3}{4}$ when $c=\frac{1}{2}(p-1)$, 
or $u=\frac{p-5}{4}$ when $c=\frac{1}{2}(p+1)$.
\end{enumerate}
\end{corollary}

Finally, we have the family of exceptional groups of Lie type,
the Suzuki groups $\mathrm{Suz}(q)$, where $q=2^{2m+1}$
and $|\mathrm{Suz}(q)|=q^2(q-1)(q^2+1)$. The relevant
conjugacy class information is contained in 
Table \ref{table:Suzuki}; this can be obtained
from \cite{MR0136646}*{Theorem 13} or 
\cite{MR0543261}*{Section 1}.
There are $q+2$ non-identity conjugacy classes: a class
of an element $D$ of order $2$, two classes of elements 
of order $4$ with representatives $H^{\pm 1}$, and three
families of classes with representatives certain powers
of elements $A,B$ and $C$ that have orders
$q-1,q+r+1$ and $q-r+1$, with $r=2^{m+1}$.
The centralizer orders are computed from the 
character table using the column orthogonality relations.
Note that $(q+r+1)(q-r+1)=q^2+1$.

\begin{table}[t]
    \centering
\begin{tabular}{ccccccc}\hline
    $\mathrm{Suz}(q)$&$D$&$H$&$H^{-1}$
    &$A^\ell\,(1\leq\ell\leq\frac{q-2}{2})$
    &$B^\ell\,(1\leq\ell\leq\frac{q+r}{4})$
    &$C^\ell\,(1\leq\ell\leq\frac{q-r}{4})$\vrule width 0mm height 4 mm depth
  2mm\\\hline
  $g$&$2$&$4$&$4$&$[q-1,\ell]$&$[q+r+1,\ell]$
  &$[q-r+1,\ell]$\vrule width 0mm height 4 mm depth
  0mm\\
    $g^G$&$Q$&$\frac{1}{2}qQ$&$\frac{1}{2}qQ$
    &$q^2(q^2+1)$&$q^2(q-1)(q-r+1)$
    &$q^2(q-1)(q+r+1)$\vrule width 0mm height 5 mm depth
  0mm\\
    $C_g$&$q^2$&$2q$&$2q$&$q-1$&$q+r+1$&$q-r+1$\vrule width 0mm height 5 mm depth
  2mm\\
    \hline
  \end{tabular}
      \caption{Conjugacy classes of the Suzuki groups 
      $\mathrm{Suz}(q)$ for 
      $q=2^{2m+1}$, where $r^2=2q$ and $Q=(q-1)(q^2+1)$.}
    \label{table:Suzuki}
\end{table}

\begin{corollary}
Let $G$ be the Suzuki groups $\mathrm{Suz}(q)$,
where $q=2^{2m+1}\,(m\geq 1)$
and let $F$ be the order sheaf on $X_G$.
Let $c=q-1$ or $q\pm r+1$
where $r=2^{m+1}$.
Then
$H_*(X_G;F)$ is non-vanishing in degrees
$0$ and $\omega(c)-1$.
\begin{enumerate}[(i).]
    \item 
Write the prime factorisation
$q-1=\prod_{\Omega(c)} s^{m(s)}$
and let $r=\frac{q-2}{2}$ when $c=q-1$
or $r=\frac{q\pm r}{4}$ when $c=q\pm r+1$.
Let $\Lambda(c)$ be the set of $\ell$ in the
range $0\leq\ell\leq r$ 
for which $\gcd(c,\ell)=c/t^{k}$
for some
$t=t(\ell)\in\Omega(c)$
and $k=k(\ell)$ with $0\leq k\leq m(t)$.
Then for 
$i=\omega(c)-1$, the homology $H_i(X_G;F)$ has summand
$\bigoplus_{\Lambda(c)}\Z^{a(\ell)}_{b(\ell)}$,
where $b(\ell)=t^{k}$ and
$a(\ell)=a(c)\prod (s^{m(s)}-2)$,
the product over the $s\in\Omega(c)$ such that $s\not=t$,
and where:
$$
a(c)=
  \left\{
\begin{array}{ll}
q^2(q^2+1),  & c=q-1,\\
q^2(q^2-1)(q\mp r+1),  & c=q\mp r+1.\\
\end{array}
\right.
$$
    \item $H_0(X_G;F)$ has summands $\Z_2^{(q-1)(q^2+1)}$
    and $\Z_4^{q(q-1)(q^2+1)}$.
\end{enumerate}
\end{corollary}

The contribution to homology made by the 
elements in the conjugacy classes of the $B^\ell$ and 
the $C^\ell$ have been amalgamated together in part (i) of the corollary.
The proof is completely analogous to the proof of 
Corollary \ref{corollary:PSL_2:characteristic2:groups}.

%%%%%%%%%%%%%%%%%%%%%%%%%%%%%%%%%%%%%%%%%%%%
%
%

%
%
%
%%%%%%%%%%%%%%%%%%%%%%%%%%%%%%%%%%%%%%%%

\end{document}